\documentclass[11pt, reqno]{amsart}
\usepackage{amsmath}
\usepackage{amsthm}
\usepackage{amssymb}
\usepackage{amscd}
\usepackage{amsbsy}
\usepackage{epsfig}
\usepackage{graphicx}
\usepackage{psfrag}
\usepackage{bbm}
\usepackage{geometry}
\usepackage[backref,bookmarks=false]{hyperref}
\usepackage[UKenglish]{babel}% http://ctan.org/pkg/babel
\usepackage[UKenglish]{isodate}% http://ctan.org/pkg/isodate

\usepackage{enumitem}
\theoremstyle{plain}% default
\newtheorem{theorem}{Theorem}[section]

\newtheorem{remark}[theorem]{Remark}
\newtheorem{proposition}[theorem]{Proposition}
\newtheorem{lemma}[theorem]{Lemma}

\theoremstyle{definition}
\newtheorem{definition}[theorem]{Definition}

\newtheorem*{lemma*}{Lemma}
\theoremstyle{definition}

\newtheorem{example}{Example}[section]
\newtheorem*{condition}{Condition}

\def\R{\ensuremath{\mathbb R}}
\def\N{\ensuremath{\mathbb N}}

\def\I{\ensuremath{{\bf 1}}}

\def\U{\ensuremath{\mathcal U}}

\def\M{\ensuremath{\mathcal M}}

\def\L{\ensuremath{\mathcal L}}

\def\P{\ensuremath{\mathbb P}}

\def\X{\mathcal{X}}
\def\ie{{\em i.e.}, }
\def\E{\mathbb E}

\def\supp{\ensuremath{\text{supp}}}

\def\eg{\ensuremath\textrm{e.g.}}

\def\eps{\varepsilon}

\numberwithin{equation}{section}

\begin{document}
\author[B.~ Zhao]{Boyuan Zhao}
\address{
Beijing International Center for Mathematical Research (BICMR), Peking University,\\ No.5 Yiheyuan Road,\\
Haidian District, \\
Beijing, China\\
100081
}
\email{boyuan.zhao@pku.edu.cn}
\title{On record behaviours of returns for interval maps}


\begin{abstract}
Given a dynamical system $(\X,T)$ and an ergodic invariant measure $\mu$, a $\mu$-typical point $x\in \X$ under $T$ visits arbitrarily small neighbourhood of $x$. We will use techniques based on extreme value theory and Laplace functional for point process to show that, for interval maps with suitable mixing and anti-clustering conditions, the limiting distributions of the record times and values of the return process $Y_j(x)=-\log d(T^jx,x)$, after possible normalisation, has Poisson-type laws. 
\end{abstract}
\maketitle
\section{Introduction}\label{sec: introduction}
Let $(\X,T,\mu)$ be a dynamical system defined on the metric space $(\X,d)$, and $\mu$ is a $T$-invariant ergodic probability measure. For a $\mu$-typical $x\in\X$, one of the most fundamental theorems in dynamics, namely the \emph{Poincar\'e recurrence theorem}, asserts that for $\mu$-almost every $x\in\X$,
$$\liminf_{n\to\infty}d(T^nx,x)=0.
$$
An alternative interpretation is, for a $\mu$-typical point, there exists a subsequence $(n_k)_k$ such that $d(T^{n_k}x,x)$ is strictly decreasing. A rich research area of dynamics is devoted to establishing various quantitative limit laws in relation to recurrences, either regarding the return time or the return distance (see, \cite{BarSou01}, \cite{Col01}, \cite{Sau04,Sau06},  \cite{KirKunPer23}, \cite{HKKP}, \cite{GupHolNic11}, \emph{etc}.) obtaining results in various forms: almost sure convergences, dynamical Borel-Cantelli dichotomy results, fractal dimension-related formulae, extreme value theory etc. 
In this paper,
we are interested in the behaviour of the record value of return distances to the initial position $x$. More precisely, let 
\begin{equation}\label{eqn: def Y_j}
    Y_j(x):=-\log d\left(x,T^jx\right).
\end{equation}
This observable is well-defined for almost every $x$ non-periodic and all $j\in\N$. Set $\tau_1=1$ and inductively define the $k$-th record time $ \tau_k:=\inf\left\{m>\tau_{k-1}:\, Y_m>Y_{\tau_{k-1}}\right\}.$
We will develop a useful approach to determine the behaviour of $Y_{\tau_k}$. In particular, we will use results based on Extreme Value Theory to prove that the distribution of $\tau_k$ and $Y_{\tau_k}$ follow a particular limit law if the system is fast-mixing and satisfies some other conditions regarding clustering short returns (see Theorem~\ref{main thm: convergence of PP} below). More precisely, under suitable conditions, we prove that 
\begin{theorem}
    $\P(Y_{\tau_k}-\log n\le r)\to\int_0^r\frac{z^{k-1}}{(k-1)!}e^{-z}dz$. Equivalently, the $k$-th record after normalisation has a Gamma distribution. 
\end{theorem}
To guide us on why we should expect such form of the limiting distribution, let us first look at the \emph{i.i.d.}\ case, since an observable on a fast mixing dynamical system behaves like a chain of \emph{i.i.d.}\ random variables after relatively few iterates.

\subsection{Motivation: independent with common distribution case}\label{ssec: iid evl}
Let $\{X_i\}_{i=0,1,\dots}$ be a sequence of \emph{i.i.d.}\ random variables taking values in $[0,1]$ with common distribution function $F$ and density  $f$. Inheriting the definition above, \emph{i.e.}\ $Y_j=-\log d(X_0,X_j)$, we observe the following:

{Condition on the event $X_0=x$,} then $(Y_j)_{j\ge1}$ is \emph{i.i.d.}\ and has distribution
$$\begin{aligned}
   G_x(y):=& \P(Y_j\le y)=\P(|X_j-x|\ge e^{-y})=1-\left(F(x+e^{-y})-F(x-e^{-y})\right)\\
   =&1-\int_{x-e^{-y}}^{x+e^{-y}}f(t)dt
\end{aligned}.$$
So $Y_j$ has density $g_x(y)=\frac d{dy}G_x(y)=e^{-y}[f(x+e^{-y})+f(x-e^{-y})]$.

By, for example \cite{Kar66,HolStr69,Res73}, using the hazard function,
\begin{equation}\label{eqn: dis of Y_tau}
    \begin{split}
        \P\left(Y_{\tau_k}\le y|X_0=x\right)&=\int_{-\infty}^{y}\frac{(-\log(1-G_x(z)))^{k-1}}{(k-1)!}dG_x(z)\\
        %&=\int_{0}^y\frac1{(k-1)!}\left(-\log\int_{x-e^{-z}}^{x+e^{-z}}f(\xi)d\xi\right)^{k-1}e^{-z}[f(x+e^{-z})+f(x-e^{-z})]dz\\
        %&=\int_{-\infty}^y\frac1{k!}-(\log f(x,y))^ke^{-y}(f(x+e^{-y})-f(x-e^{-y}))dy
    \end{split}
\end{equation}
Note that $\left(F(x+e^{-z})-F(x-e^{-z})\right)=\int_{x-e^{-z}}^{x+e^{-z}}f(\xi)d\xi$ is comparable with the probability of visiting an $r$-neighbourhood around $x$ with $r=e^{-z}$ (in dynamical system context, for a system $(\X,T,\mu)$ this is simply $\mu(B(x,e^{-z}))$). By the Lebesgue differentiation theorem, $F(x+e^{-z})-F(x-e^{-z})\approx 2f(x)e^{-z}$, $G_x(y)\approx 1-2f(x)e^{-y}$ and $Y_{\tau_k}-\log (2f(x))$ is asymptotically $\Gamma(k,1)$ (actually it is the special case of Gamma distribution, known as Erlang, since $k\in \N$) distributed, so, integrate over $x$ according to $X_0$,
$$
\P\left(Y_{\tau_k}-\log(2f(x))\le y\right)\approx\int_{-\infty}^{\infty}\int^{y}_{0}\frac {t^{k-1}}{(k-1)!}e^{-t}\, f(x)\ dt\ dx.
$$
The proof of records for \emph{i.i.d.}\ random variables follows from proving that the $k$-th record has the same distribution of the sum of $k$ number of \emph{i.i.d.}\ random variables with the exact same distribution, plus the fact that the independent increments between record has exponential distribution (see $\eg$ \cite{Res73}). We can establish the result via extreme value theory and point process in a rigorous way.

In classical extreme value theory, given a sequence of \emph{i.i.d.}\ random variables $\{\xi_i\}_i$ with common distribution $F$, if for each $\tau\ge0$, there is $u_n(\tau)$ a sequence of numbers such that
\begin{equation}\label{eqn: Leadbetter u_n}
    n(1-F(u_n(\tau)))\to\tau \text{ as $n\to\infty$},
\end{equation}
then (see for example \cite{LLR80}) the partial maxima $M_n=\max_{i\le n}\xi_i$
\begin{equation}\label{eqn: Leadbetter M_n convergence}
\P\left(M_n\le u_n(\tau)\right)\to e^{-\tau} \text{  in distribution as $n\to\infty$}.
\end{equation}
Thus, coming back to our previous question, again conditioning on $X_0=x$ again and set $u_n(\tau)=\tau+\log n$, 
\begin{align*}
    n(1-G_x(u_n))&=n \int_{x-e^{-\tau}/n}^{x+e^{-\tau}/n}f(t)dt=\frac{n2e^{-\tau}}{|B\left(x,e^{-\tau}/n\right)|}\int_{B\left(x,e^{-\tau}/n\right)}f(t)dt\\
    &\xrightarrow[n\to\infty]{}2e^{-\tau}f(x),
\end{align*}
where the last line holds almost surely by the Lebesgue differentiation theorem. This means, conditioning on $X_0=x$, the process 
$$M_n(t)=\left\{
\begin{aligned}&\max_{j\in[1,[tn]]}Y_j-\log n & \quad 1/n<t\le 1,\\
&Y_1-\log n &\quad 0<t\le 1/n.
\end{aligned}\right.$$ 
converges to an extremal process $Y(t)$ generated by the distribution $H_x(y)=e^{-2f(x)e^{-y}}$ (see \cite{ResRub73}). So integrating over the conditional event $\{X_0=x\}$,
\begin{equation}\label{eqn: iid maxima convergence} \P\left(\max_{1\le j\le n}Y_j\le \tau+\log n\right)\to \int e^{-2f(x)e^{-\tau}}f(x)dx.
\end{equation}

\begin{remark}\label{rmk: three typs of EVL}
    In classical extreme value theory the observable $Y_j$ may be defined more generally by $\varphi(d(x,X_j))$ where $\varphi$ may take one of the three form, such their normalised partial maximum converges weakly to Gumbel, Frechet and Weibull distribution, respectively:
    \begin{itemize}[noitemsep]
        \item $\varphi(t)=-\log(t)$, then $M_n(t)=M_{[nt]}-\log n$ converges weakly to an extremal process of Gumbel type;
        \item $\varphi(t)=t^{-\alpha}$ for some $\alpha>0$, then $M_n(t)=n^{-\alpha}M_{[nt]}$ converges to an extremal process of Fr\'echet type;
        \item $\varphi(t)=C-t^\alpha$ for some $C\in\R$ and $\alpha>0$, then $M_n(t)=n^\alpha\left(M_{[nt]}-C\right)$ converges weakly to an extremal process of Weibull type
    \end{itemize}
\end{remark}
For the scope of this paper, we will concentrate on the case corresponding to the Gumbel-type distribution. 
\begin{lemma}\label{lemma: Poisson and Gamma}
    Let $Z$ be a Poisson random variable with mean $\lambda$, then 
    $$\P(Z\ge k)=\int_0^\lambda \frac1{(k-1)!}t^{k-1}e^{-t}dt=\P(\Gamma(k,1)\le \lambda)$$
    where $\Gamma(k,\alpha)$ is a \emph{Gamma} distributed random variable with parameters $k$ and $\alpha$.
\end{lemma}
\begin{proof}
    Let $g(\lambda):=\P(Z\ge k)=e^{-\lambda}\sum_{j=k}^{\infty}\frac1{j!}\lambda^j$. It is clear that $g(0)=0$, and $g$ is differentiable in $\lambda$ for all $k\ge1$,  hence $g(\lambda)=\int_0^\lambda g'(t)dt$ where
    \begin{equation*}
    g'(\lambda)=\frac d{d\lambda}e^{-\lambda}\sum_{j=k}^\infty \frac1{j!}\lambda^j=-e^{-\lambda}\sum_{j=k}^\infty \frac1{j!}\lambda^j+e^{-\lambda}\sum_{j=k}^\infty\frac1{(j-1)!}\lambda^{j-1}=\frac{1}{(k-1)!}e^{-\lambda}\lambda^{k-1}.
    \qedhere\end{equation*}
\end{proof}
Recall the definitions of $\tau_k$, conditioning on $X_0=x$ again, we define the counting measures
$$R_n=\sum_{k=1}^{\infty}\delta_{\tau_k/n},\hspace{5mm} 
V_n:=\sum_{k=1}^{\infty}\delta_{\{(Y_{\tau_k}-\log n)/2f(x)\}}
$$
By \cite[Theorem 4, Theorem 5]{ResRub73}, \cite[Corollary 4.22, Corollary 4.23]{Res87} or \cite[Theorem 5.7.1, Theorem 5.8.1]{LLR80}, 
\begin{equation}
\label{eqn: weak convergence Resnick}
R_n\Rightarrow \mathcal R \hspace{5mm}\&\hspace{5mm}V_n\Rightarrow \mathcal V \hspace{3mm}\text{ both in $M_p(0,\infty)$,}
\end{equation} 
where $\mathcal R$ is a Poisson Random Measure (PRM) with intensity measure, $t^{-1} dt$, and $\mathcal V$ a PRM with intensity determined by $S(y)=y$. Recall the Gamma density function,  $\gamma_{k,1}(t)=\frac1{\Gamma(k)}t^{k-1}e^{-t}$ where $\Gamma(k)=(k-1)!$,
then for all $y>0$:
\begin{align*}\P((Y_{\tau_k}-\log n)/2f(x)\le y)&=\P(V_n(0,y]\ge k)\to \P(Poi(y)\ge k)\\
&=\int_{0}^{y}\frac1{\Gamma(k)}t^{k-1}e^{-t}dt=\P(\Gamma(k,1)\le y),
\end{align*} 
In the above formula we suppressed the dependence on $X_0=x$, so after averaging over $X_0=x$,
\begin{equation}\label{eqn: iid record value convergence}
\P\left(d\left(X_0,X_{\tau_k}\right)\ge e^{-2f(x)y}/n\right)\to\int\P\left(\Gamma(k,1)\le y\right)f(x) dx
\end{equation}

The above result is comparable with \eqref{eqn: dis of Y_tau}, with a caveat that $\tau_k$ represents the $k$-th record above the threshold $\log n$. Moreover, from \eqref{eqn: iid record value convergence} one deduces that due to the conditioning on initial condition $X_0$, the distribution of the record values will be averaged according to the density $f$ (in fact, $X_0$ needs not to have the same distribution as the other $X_i$'s as long as $X_0$ is defined on the same probability space), which is similar to introducing another level of stochasticity. We can interpret this via the perspective of a \emph{Cox process}.

\subsection{Cox process}\label{ssec: iid Cox}
In this subsection we show that the two-dimensional point process associated to the return value process \eqref{eqn: def Y_j} defined for \emph{i.i.d.}\ random variables form a  \emph{Cox process}, which we describe below.
\begin{definition}\label{defn: Cox proc}    
Let $E$ be a metric space. A \emph{Cox process} on $E$ is a point process on $E$ such that its distribution is a mixture of Poisson processes. In particular, its \textbf{random intensity measure} $\Lambda$ defined on the state space \(E\) (e.g.\ \(\R^d\) or \([0,\infty)\)) is given by the following:
\begin{itemize}
    \item a random measure $\Lambda$ is drawn according to some distribution.
    \item Then, conditional on $\Lambda = \lambda$, the point process $N$ is an inhomogeneous Poisson process with intensity measure $\lambda$.
\end{itemize}
\end{definition}
In simple terms, a Cox process is a doubly stochastic Poisson point process. It was first properly described in (hence named after the author of) \cite{Cox55}, and is well studied for a wide range of applied problems including biomathematics and flock modelling. We refer to \cite[\S13]{LasPen17} for a thorough discussion of its various aspects of statistical properties. The Cox process is also particularly relevant for our return statistics governed by dynamical systems: the point process for a hitting to a small neighbourhood of a fixed initial condition converges weakly to a Poisson process (see \S\ref{ssec: HT19 result} below), then heuristically, varying the initial position causes the `double' stochasticity. 

The \textbf{Laplace functional} of a Cox process is
\begin{equation}\label{eq:laplace-cox}
    \E\left[e^{-\int g \, dN}\right] = \E\!\left[\exp\!\left(-\int_E (1-e^{-g(z)})\,\Lambda_x(dz)\right)\right],
\end{equation}
for all non-negative continuous $g$ on the state space with compact support. By \cite[Theorem 4.2]{Kal83}, to prove convergence of point measures associated to stochastic processes, it is enough to prove the convergence of Laplace functionals for such test functions. Then we can recover the convergence described by \eqref{eqn: iid maxima convergence} via the following computation.
\begin{proposition}\label{prop: iid cox}
The point process $$N_n:=\sum_{j=1}^n\delta_{(j/n,Y_j-\log n)}$$
    $N_n\Rightarrow \mathcal N_{cox}$ where $\mathcal N_{cox}$ is a Cox point process with random intensity measure $\Lambda_x(dtdy)=2f(x)e^{-y}dydt$, and $\Lambda_x$ is drawn with density $f(x)$.
\end{proposition}
\begin{remark}Note that we recover the limit in \eqref{eqn: iid maxima convergence} by the following:
\begin{align*}&
\lim_n\P(M_n-\log n\le \tau)=\lim_n\P(N_n((0,1]\times(\tau,\infty)=0)\\
&=\P(\mathcal N_{cox}((0,1]\times(\tau,\infty)=0)=\int \ e^{-\int_0^1\int_{\tau}^\infty 2f(x)e^{-y}dydt}f(x)dx.
\end{align*}
\end{remark}

We need the following lemma to prove such convergence.
\begin{lemma}\label{lemma: Cesaro average converges to a implies product converges to exp(-a)}
    Let $\{a_j\}_{j\in\N}$ be a sequence of non-negative real numbers. Then, if $\sup_{j\in\N}a_j\le M$ for some $M<\infty$, and $\frac1n\sum_{j=1}^n a_j\to a$ for some $a\in \R_{\ge0}$, then 
    $$
    \lim_{n\to\infty}\prod_{j=1}^n\left(1-\frac1n a_j\right) = 
    \lim_{n\to\infty}e^{-\frac1n\sum_{j=1}^n a_j}=e^{-a}.
    $$
\end{lemma}
\begin{proof}
    As $|a_j|\le M$ for all $j$, then for big $n$, there is $|a_j/n|\le \frac 12$ for all $j$. By calculating the MacLaurin series of $\log (1-t)$, there is some $M'<\infty$ such that for all $0\le t\le \frac12$,
    $$
    \left|t+\log (1-t)\right|= \left|t-t-\frac{t^2}{2}-\frac{t^3}3-\cdots\right|\le M' t^2.
    $$
Therefore,
\begin{equation*}
    \log\prod_{j=1}^n\left(1-\frac1n a_j\right)=\sum_{j=1}^n\log\left(1-\frac1n a_j\right)=-\frac1n\sum_{j=1}^n a_j\pm M'\sum_{j=1}^n\left(\frac1n a_j\right)^2.\qedhere
\end{equation*}
\end{proof}
\begin{proof}[Proof of Proposition~\ref{prop: iid cox}]
    We will calculate the Laplace transform of $N_n$, as in \cite[Proposition 3.5]{Res87} tells us that the Laplace transform uniquely determines a point measure. Let $g$ be a non-negative continuous function with compact support (without loss of generality suppose the support of $g$ on the $z$ entry is contained in $[a,b)\subset \R_{\ge0}$).% and let $N_n^x$ be the point measure corresponding to the event $\{X_0=x\}$. 
    Then, by independence and identical distribution,

$$
    \E\left[e^{-\int g\,dN_n}\right] =\E\left[e^{-\sum_{i=1}^ng(i/n,Y_i-\log n)}\right]= \int\left(\prod_{i=1}^n\E\left[e^{-g(i/n,Y_i-\log n)}|X_0=x\right]\right) \ f(x)dx.
$$
Now we continue to calculate $\E\left[e^{-g(i/n,Y_i-\log n)}\right]$, which is always integrable due to continuity and compact support. Let $Z_n=Y_1-\log n$, then $F_{Z_n}=\P(Z_n<z)=\P(|X_1-x|>e^{-z}/n)=1-\int_{x\pm e^{-z}/n}f(\xi)d\xi$. So $$f_{Z_n}(z)=e^{-(z+\log n)}\left[f(x+e^{-(z+\log n)})+f(x-e^{-(z+\log n)})\right].$$
Therefore, 
\begin{align*}\E\left[e^{-g(i/n,Y_1-\log n)}|X_0=x\right]&=1-\E\left[1-e^{-g(i/n,Y_1-\log n)}|X_0=x\right]\\
&=1-\int (1-e^{-g(i/n,z)})f_{Z_n}(z) \ dz\\
&=1-\frac1n \int \left[f(x+e^{-z}/n)+f(x-e^{-z}/n)\right](1-e^{-g(i/n,z)})e^{-z}dz\\
&=:1-\frac1n A_{i,n}(x). 
\end{align*}
By the Lebesgue Differentiation Theorem, for almost every $x$, $f(x\pm e^{-z}/n)\to f(x)$ as $n\to\infty$, hence 
$$\lim_{n\to\infty}\frac 1n\sum_{i=1}^n A_{i,n}(x)=2 f(x)\iint\left(1-e^{-g(t,z)}\right)e^{-z}dzdt$$

By Lemma~\ref{lemma: Cesaro average converges to a implies product converges to exp(-a)}, 
\begin{align*}
\lim_{n\to\infty}\E\left[e^{-\int g\,dN_n}|X_0=x\right]= \lim_{n\to\infty}\prod_{i=1}^n\left(1-\frac1n A_{i,n}\right)=\exp\left(-2f(x)\iint\left(1-e^{-g(t,z)}\right)e^{-z}dzdt \right)
\end{align*}
Lastly, applying the Dominated Convergence Theorem,
\begin{align*}
    \lim_{n\to\infty}\E\left[e^{-\int g\,dN_n}\right]&=\int\lim_{n\to\infty}\E\left[e^{-\int g\,dN_n}|X_0=x\right]f(x)dx\\
    &=\int\exp\left(-2f(x)\iint\left(1-e^{-g(t,z)}\right)e^{-z}dzdt \right)f(x)dx\qedhere
\end{align*}
\end{proof}
The beauty in establishing the convergence of $N_n$ is that we can then apply the record-extraction mappings to $N_n$ and $\mathcal N$ then apply the Continuous Mapping Theorem after checking the necessary almost sure continuity conditions, then the analogous formula as \eqref{eqn: weak convergence Resnick} can be obtained for the return process for the \emph{i.i.d.}\ process. Especially, the particular proof for deriving such convergence for record time/value processes does not require conditioning argument. We will formally describe this in dynamical system context in \S~\ref{sec: proof of record point process convergence}.

\subsection{Hitting problem in dynamical setting}\label{ssec: HT19 result}
From our observations in the \emph{i.i.d.}\ quantitative recurrence case, one sees that the proof was based on conditioning on the initial position $X_0=x_0$, after which the remaining calculations behave like a ball-hitting problem, and we apply independence among the random variables. Then, using EVL results for stationary processes followed by integrating over the whole space, one obtains an averaged Gumbel distribution for the partial maxima process and  an averaged-Gamma distribution for the records after normalisation. For dynamical systems, to convince ourselves that the record values do converge to the conjectured averaged-Gamma distribution, it is useful and reassuring to verify if the point measures associated with a hitting process with a fixed reference point still converges to a PRM. Indeed, such result holds.

Fix a generic $\tilde x\in\X$ and set $Y^{\tilde x}_j=-\log d(f^jx,\tilde x)$, also the partial maximum $M_j^{\tilde x }=\max\{Y^{\tilde x}_1,\dots,Y^{\tilde x}_j\}$. For $\mu$-almost every $\tilde x$, $\{Y^{\tilde x}_j \}_j$ is a stationary process, and by Lebesgue differentiation theorem, $n\mu\left(Y_j^{\tilde x}>\tau+\log n\right)\to2\rho(\tilde x)e^{-\tau}$ if $\rho$ is bounded. By \cite[Theorem 2.3, Theorem 3.1]{HolTod19}, under suitable mixing and recurrence conditions (the paper in question called them $\mathcal D_r(u_n^{(k)})$ and $\mathcal D'_r(u_n^{(k)})$, motivated by Leadbetter's works) satisfied by a wide range of maps enjoying fast decay of correlations, the process
$$
M^{\tilde x}_n(t):=\left(M^{\tilde x}_{[nt]}-\log n \right)
$$
converges to an extremal-$F$ process  $\mathbbm D(0,\infty)$ (with respect to the Shorokhod $J_1$ topology, see \cite[\S1]{HolTod19} for more details) generated by the distribution $F(u)=e^{-2\rho(\tilde x)e^{-u}}$. Moreover, the counting measures 
$$
R_n=\sum_{j=1}^{\infty}\delta_{j/n}\cdot \I_{\left\{Y^{\tilde x}_j >M_{j-1}^{\tilde x}\right\}},\hspace{5mm} 
V_n:=\sum_{k=1}^{\infty}\delta_{\left(M^{\tilde x}_n({\tau_k/n})\right)}
$$
both converge in $M_p(0,\infty)$ to PRM with intensity $1/t\,dt$ and $\lambda[a,b]=b-a$ respectively. In this case, 
\begin{align*}
    \lim_{n\to\infty}\P\left(\left(Y^{\tilde x}_{\tau_k}-\log n\right)\le s\right)&=\lim_{n\to\infty}\left(V_n[0,s]\ge k\right)=\sum_{j=k}^{\infty}\frac1{j!}s^je^{-s}\\
    &=\int_0^s\frac1{(k-1)!}t^{k-1}e^{-t}\,dt=\P(\Gamma(k,1)\le s)
\end{align*}
So the normalised record values (that are above 0) do follow a Gamma distribution in limit\footnote{See Remark~\ref{rmk: discuss after main reults} for further discussion}.

\subsection{Statistics of records for recurrence in dynamics}
The stationary nature of the problem can no longer directly facilitate the analysis when we consider the return process $Y_j(x)=-\log d(T^jx,x)$ for a dynamical system. However, various remarkable progress regarding EVL type convergence has been made in this non-stationary setting.
For dynamical systems satisfying Young's condition, \cite[Theorem 4.1]{Col01}established a formula similar to \eqref{eqn: iid maxima convergence}, and in \cite[Theorem 2.3]{HolTod25}, for suitable choices of $u_n$, the number of returns to $B(x,e^{-u_n})$ converges to a Poisson-like distribution. We give a brief introduction to their results.

Let $(\X,T,\mu)$ be a dynamical system over the unit interval with domain of $T=\X\subset[0,1]$
that satisfies the following assumptions.
\begin{itemize}
    \item[(A1)]\label{cond: A1} $\mu$ is a $T$-invariant probability that has exponential decay of correlations for $BV$ against $L^1(\mu)$ observables, that is, there exists $C_T,\eta>0$ for all $f\in \mathcal{BV}$ and $g\in L^1(\mu)$, for all $k\in\N$,
    $$\left|\int f\cdot g\circ T^kd\mu-\int f\,d\mu\int g\,d\mu\right|\le C_T\|f\|_{BV}\|g\|_{L^1(\mu)}e^{-\eta k},$$
    where for a function $\varphi:\X\to \R$,
    $\|\varphi\|_{BV}=\|\varphi\|_{\infty}+TV(\varphi).$
    
    \item[(A2)]\label{cond: A2} suppose $r=r_n\asymp n^{-a}$ for some $a>0$. Then there exists $C_0,\beta_0>0$ such that for $E_j(r):=\{x\in\X:d(x,T^jx)<r\}$ and all $j\le (\log n)^2$, $\mu(E_j(r))\le C_0r^{\beta_0}$. Consequently, there exists $C_1,\beta_1>0$ such that $\mu(\widetilde E(r))\le C_1r^{\beta_1}$ where $\widetilde E(r)=\cup_{j=1}^{(\log n)^2}E_j(r)$.
\end{itemize}
In the language of \cite{HolTod25}, $(\X,T,\mu)$ is said to be \emph{hyperbolic recurrent} if it satisfies (A1) (A2) and $\mu$ is an invariant ergodic measure with density $\rho\in BV$. This implies $\|\rho\|_\infty<\infty$.
\begin{theorem}\label{thm: HT25}
    For a hyperbolic recurrent dynamical system (\cite[Theorem 2.3]{HolTod25}) $(\X,T,\mu)$, or a $C^2$ interval map satisfying Young's condition (\cite[Theorem 4.1]{Col01}) , 
    \begin{equation*}\label{eqn: averaged Gumbel law}
        \lim_{n\to\infty}\P\left(\max_{1\le j\le n}Y_j(x)\le \tau+\log n\right)=\int e^{-2\rho(x)e^{-\tau}}\rho(x)\,dx
    \end{equation*}
    \begin{equation*}\label{eqn: averaged Poisson law}
        \lim_{n\to\infty}\P\left(\#\left\{j\le n:d(T^jx,x)\le e^{-\tau}/n\right\}=k\right)\to\int\frac{(2\rho(x)e^{-\tau})^{k}}{k!}e^{-2\rho(x)e^{-\tau}}\rho(x)dx
    \end{equation*}
\end{theorem}

The above formulas say that the among the first $n$ iterates, the largest log value of return distance to initial position converges to an averaged Gumbel law after normalisation, similar to the case in \S\ref{ssec: iid evl}. In this paper we wish to obtain a more detailed statistics for returns; in particular, our aim is to verify limiting laws similar to \eqref{eqn: weak convergence Resnick}, which encode information for the record behaviours of returns. We will now describe the class of maps for which our mains theorem apply.

\subsection*{Notations}
We will use the following notations. For a set or an event $E$, its indicator function is denoted by $\I_E$. For any real numbers $A,B,C$, write $A=B\pm C$ if $B-C\le A\le B+C$; for any sequence $a_n$, denote $a_n\lesssim n^{-\alpha}$ or $a_n=\mathcal O(n^{-\alpha})$ if there exists a uniform constant $K$ such that $a_n\le Kn^{-\alpha}$ for all $n$ large enough, and $a_n\asymp b_n$ if $a_n\lesssim b_n$ and $a_n\gtrsim b_n$.

\section{Main results}\label{sec: main results}
We now describe the precise setting of the maps our main results apply to. Let $\mathcal{I}$ be a countable index set, and $\mathcal{P}=\{P_a\}_{a\in\mathcal{I}}$ a collection of subintervals in $\X$ with disjoint interiors covering $\X$. We say $T$ is a \textit{uniformly expanding Markov map} if\\
(1) for any $a\in\mathcal I$, $T_a:=T|_{P_a}$ is continuous, injective and $T(P_a)$ is a union of elements in $\mathcal P$; \\
(2) there is a uniform constant $\upsilon>1$ such that for all $a\in\mathcal I$, $|DT_a|\ge\upsilon$. 

The \emph{repeller} of $T$, denoted by $\Xi$, is the collection of points with all their forward iterates contained in $\mathcal P$, namely
    $$
    \Xi:=\left\{x\in\X:T^k(x)\in\bigcup_{a\in\mathcal I}P_a\text{ for all $k\ge0$}\right\}.
    $$
The invariant ergodic measure $\mu$ in question is supported on $\Xi$. There is a shift system associated to $T$: let $M$ be an $\mathcal{I}\times \mathcal{I}$ matrix such that $M_{ab}=1$ if $T(P_a)\cap P_b\neq \emptyset$ and 0 otherwise. Let $\Sigma$ denote the space of all \emph{infinite admissible words}, \ie
$$
\Sigma:=\left\{{x}=(x_0,x_1,\dots)\in\mathcal I^{\N_0}:M_{x_k,x_{k+1}}=1,\,\forall \,k\ge0\right\}.
$$
The natural projection map $\pi:\Sigma\to\Xi$ is defined by
$$
{x}=\pi\left(x_0,x_1,\dots\right) \text{ if and only if } x\in\bigcap_{i=0}^{\infty}T^{-i}P_{x_i}.
$$ 
The dynamics on $\Sigma$ is the left shift $\sigma:\Sigma\to\Sigma$ given by $\sigma(x_0,x_1,\dots)=(x_1,x_2,\dots)$, then $\pi$ defines a semi-conjugacy $T\circ\pi=\pi\circ\sigma$. The corresponding symbolic measure $\tilde\mu$ of $\mu$ is given by $\mu=\pi_*\Tilde\mu$, \ie for all Borel-measurable set $B\in\mathcal B([0,1])$, $\mu(B)=\tilde\mu\left(\pi^{-1}B\right)$.

Denote $\mathcal{P}^n:=\bigvee_{j=0}^{n-1}f^{-j}\mathcal{P}$, each $P\in\mathcal P^n$ corresponds to an \emph{$n-$cylinder} in $\Sigma$: let $\Sigma_n\subseteq \mathcal I^n$ denote all finite words of length $n$ and for any $\textbf{i}\subset\Sigma$, the $n$-cylinder defined by $\textbf{i}$ is 
$$
[\textbf{i}]=[i_0,\dots,i_{n-1}]:=\left\{y\in\Sigma:y_j=i_j,\,j=0,\dots,n-1\right\}.
$$
Then $\pi[i_0,i_1,\dots,i_{n-1}]=\bigcap_{j=0}^{n-1}T^{-j}P_{i_j}=:P_{\textbf{i}}$. The \emph{depth} of a cylinder $[\textbf{i}]$ is the length of $\textbf{i}$.

\subsection{Technical assumptions on $(\X,T,\mu)$}
In order to resolve the issue created by non-stationarity of our return process due to the independence of the initial position, logically we will require our system to exhibit fast decay of correlations so that the return process $Y_j(x)=-\log d(x,T^jx)$ resembles an \emph{i.i.d.}\ sequence after relatively few iterates. In order to make sure the records do not accumulate, we also need to require the set of points that returns to a small neighbourhood of themselves is not very large so that records do not accumulate in a small time interval near 0 after rescaling. Therefore, in this paper, we would focus on systems that are hyperbolic recurrent (see \S\ref{ssec: HT19 result}) as (A1) and (A2) are natural candidates of assumptions that meet our requirements. In particular, \hyperref[cond: A2]{(A2)} implies the following `averaged' version.

\begin{lemma}\label{lemma: A2 implies A2+}
     Let $(\X,T,\mu)$ be a dynamical system satisfying \hyperref[cond: A2]{(A2)} and $\rho\in\mathcal {BV}$. Then for every $r_n\lesssim n^{-1}$, 
$$
n\int\sum_{j=1}^{(\log n)^2}\mu\left(B(\zeta,r_n)\cap T^{-j}B(\zeta,r_n)\right)d\mu(\zeta)\lesssim \frac1{\log n}.
$$
\end{lemma}
\begin{proof}
For each $j\le (\log n)^2$,    
    \begin{align*}
      \int\mu\left(B(\zeta,r_n)\cap T^{-j}B(\zeta,r_n)\right)d\mu(\zeta)=& \iint \I_{B(\zeta, r_n)}(x) \cdot \I_{B(\zeta, r_n)}(T^j x) \, d\mu(x) \, d\mu(\zeta)\\
      =&\int\left[\int\I_{B(\zeta,r_n)}(x)\cdot \I_{B(\zeta,r_n)}(T^jx)d\mu(\zeta)\right]d\mu(x),
    \end{align*}
where the second line follows from Fubini. Notice that 
$$\I_{B(\zeta,r_n)}(x)\cdot \I_{B(\zeta,r_n)}(T^jx)=\I_{B(x,r_n)}(\zeta)\I_{B(T^jx,r_n)}(\zeta),
$$ 
so the inner integral equals $\mu\left(B(x,r_n)\cap B(T^jx,r_n)\right)$. The intersection is non-empty only when $x\in E_j(2r_n)$, and in this case we can trivially bound $\mu\left(B(x,r_n)\cap B(T^jx,r_n)\right)$ by $\mu(B(x,r_n))\lesssim r_n$ since $\|\rho\|_\infty<\infty$. Then, by \hyperref[cond: A2]{(A2)} we can bound the original integral by
\begin{equation*}
          \int\mu\left(B(\zeta,r_n)\cap T^{-j}B(\zeta,r_n)\right)d\mu(\zeta)\le \int_{E_j(2r_n)}\mu(B(x,r_n))d\mu(x)\lesssim r_n^{\beta_0+1}.
\end{equation*}
Summing over $j\le (\log n)^2$, for $r_n\asymp n^{-1}$, as $\beta_0>0$, for all $n$ large enough, 
\begin{equation*}
n(\log n)^2r_n^{\beta_0+1}\asymp (\log n)^2n^{-\beta_0}\lesssim \frac1{\log n}.\qedhere
\end{equation*}
\end{proof}

The next two assumptions will be helpful to prove the rate of convergence of point measures associated with the process $Y_j$. The first one is on the invariant density $\rho$ and the second is on the corresponding symbolic dynamic of $(\X,T)$.
\begin{condition}[LC]\label{cond: LC}
    The density $\rho$ is said to satisfy (LC) if there exists a closed set $\X_0\subseteq \X$ with $\mu(\X_0)=1$ and a uniform constant $C_\rho$ such that for  every point $x\in \X_0$, for all $|h|$ small enough,
    $$\left|\rho(x+h)-\rho(x)\right|\le \frac{C_\rho}{|\log|h||}.
    $$
\end{condition}
One will see as the proof proceeds, that the logarithmic part of the assumption will only be used once to overcome a technical difficulty regarding convergence rate. If the proof techniques are improved to bypass such difficulty, then the logarithmic continuity can be removed.

\begin{condition}[A1+]\label{cond: A1+}
    Say $\mu$ is \emph{exponentially $\psi$-mixing} if there are $\tilde C_0,\eta'>0$ and a monotone decreasing function $\psi(k)\le \tilde C_0e^{-\eta'k}$ for all $k\in\N$, such that the corresponding symbolic measure $\Tilde\mu$ satisfies: for all $n,k\in \N$, $\textbf{i}\in\Sigma_n$ and $\textbf{j}\in\Sigma^*=\bigcup_{l\ge1}\Sigma_l$,
    $$\left|\frac{\Tilde\mu([\textbf{i}]\cap\sigma^{-(n+k)}[\textbf{j}])}{\Tilde\mu[\textbf{i}]\Tilde\mu[\textbf{j}]}-1\right|\le \psi(k).$$
\end{condition}
 The $\psi$-mixing condition is rather strong, in the sense that it implies $\phi$-mixing, $\beta$-mixing and other weaker mixing conditions. However, it is satisfied by a wide range of interval maps with a Gibbs-Markov structure.

\subsection{Main result}
Similar to \S~\ref{ssec: HT19 result}, define the partial maximum $M_j:=\max\{Y_1,\dots,Y_j\}$ and the normalised partial maximal process 
$$
M_n(t):=M_{[tn]}-\log n=\max_{1\le j\le [nt]}-\log d(T^jx,x)-\log n.
$$ 
Recall from \S\ref{sec: introduction} the definition of $\tau_k$. It is obvious that $Y_{\tau_k}-\log n=M_n(\tau_k/n)$. Also define the point processes
\begin{equation}\label{eqn: def of PP for return}
    R_n:=\sum_{j=1}^\infty \delta_{j/n}\cdot \I_{\{Y_j>M_{j-1}\}} \hspace{5mm} 
    V_n:=\sum_{\tau_k}\delta_{M_n(\tau_k/n)}
\end{equation}

\begin{theorem}\label{main thm: convergence of PP}
    Assume that $(\X,T,\mu)$ is a piecewise uniformly expanding interval map that is hyperbolic recurrent, and satisfies \hyperref[cond: A1+]{(A1+)} and (LC).  Then 
    \begin{itemize}
        \item[(1)] $R_n$ converges weakly to a Poisson random measure $\mathcal R$ on $(0,\infty)$ with intensity $1/tdt$, and
        \item[(2)] $V_n$ also converges weakly to a Poisson random measure $\mathcal V$ on $(0,\infty)$ with intensity measure $\lambda_V[a,b]=(b-a)$.
    \end{itemize}
\end{theorem}

The above theorem is a consequence of the following:
\begin{theorem}\label{thm: convergence to Cox}
    Under the assumption of Theorem~\ref{main thm: convergence of PP}, the point process $$N_n:=\sum_{j=1}^n\delta_{\left(j/n,Y_j-\log n\right)}$$ 
    converges weakly to a \hyperref[defn: Cox proc]{Cox process} $\mathcal N$ on $(0,1]\times\R$ with random intensity measure given by $\Lambda_x(dydt)=2\rho(x)e^{-y}dydt$, while $\Lambda_x$ is drawn over $x\in \X$ with density $\rho(x)$. 
\end{theorem}
As far as the author knows, although the rate of convergence (with error $\lesssim 1/\log n$) is not directly derived for the distribution, our result is still one of the few quantitative statement in the literature of dynamical extreme value theory, and is first in establishing a relation between return process and the Cox process. To upgrade to a genuine rate of convergence of the laws may require other techniques such as Stein's method using characterising equations. We also make the following remarks.
\begin{remark}\label{rmk: discuss after main reults}
(1) For intuition, conditional on the initial point $x$, the limiting point
process is Poisson with intensity $2\rho(x)e^{-z}\,dt\,dz$ which coincides with the hitting result in \cite{HolTod19}, then averaging over the initial point therefore produces the Cox process which agrees with the main theorem \cite{HolTod25}.

(2) One may be tempted to conclude, as a consequence of  Theorem~\ref{main thm: convergence of PP}, that for fixed $k$, the $k$-th record also follows a $\Gamma(k,1)$ distribution. Indeed, by Lemma~\ref{lemma: Poisson and Gamma}, the probability of having more than $k$ records in the interval $[0,r]$ is
$$
\P(V_n[0,r]\ge k)\to \P(\Gamma(k,1)\le r).
$$
This is consistent with \S~\ref{ssec: iid evl}. This appears to be unreasonable: for any fixed $k$, the ordinary $k$-th record is almost surely finite, hence the normalisation $Y_{\tau_k}-\log n$ goes to $-\infty$ as $n\to\infty$ almost surely. The correct way to interpret the main theorem is: the $k$-th record above the $\log n$ threshold has a $\Gamma(k,1)$ distribution. In other words, if we set 
$$J_n:=\#\left\{j\le n:Y_j>M_{j-1}\text{ and }Y_j<\log n\right\},
$$ then for any $k\in\N$, $\P\left(Y_{\tau_{J_n+k}}-\log n\le r\right)\to \P(\Gamma(k,1)\le r)$. This interpretation is consistent with motivating results in \emph{i.i.d.}\ setting, and reflects the convergence in vague topology: such topology only requires convergence for test functions with compact support, hence the values diverge to $-\infty$ after recentering are simply not seen by the test function.

(3)  By \cite[Proposition 4.12, Proposition 4.13]{Res87}, for the \emph{i.i.d.}\ case as in \S~\ref{ssec: iid evl}, as a consequence of \cite[Proposition 4.11]{Res87}, the Continuous Mapping Theorem \cite{Kal83} and embedding into a extremal process, for almost every $x$,
\begin{equation}\label{eqn: iid record time convergence}
\frac{\log \tau_k}{k}\to1\text{ and } \frac{\log(\tau_k-\tau_{k-1})}{k}\to 1.
\end{equation}
Due to the strongly dependent nature of our dynamical process, the analogous almost sure statement has yet been derived because explicit error bounds of convergence rate are hard to obtain. But it is very possible that these formulas hold. Especially, if there is a way to upgrade the rate of convergence (currently at $1/\log n$) of Laplace functionals to a stronger estimate (so that the error terms are summable in $n$), then one may be able to apply Borel-Cantelli arguments.
\end{remark}

\section{Examples}
We give examples of interval maps that satisfy the assumptions for our main results. In particular, these maps have the Gibbs-Markov property hence the corresponding symbolic dynamics is a subshift of finite type or full-shift on countable alphabet, and the corresponding invariant measure $\tilde \mu$ is the unique Gibbs equilibrium for the geometric potential $-\log|T'\circ\pi|$. $\tilde \mu$ is $\psi$-mixing \cite{Bow75}. Condition $(A2)$ can be verified by \cite[Lemma 3.4]{HolNicTor12}\footnote{The statement of \cite[Lemma 3.4]{HolNicTor12} is a bound for the Lebesgue measure of $E_j(r)$ sets, but since we will take $\mu$ which is an acip, the upper bound holds with the same exponent for $\mu(E_j(r)$.} 
\begin{example}[Doubling map]
    Let $T(x)=2 x \, (\text{mod }1)$ be the doubling map, and $\mu=\mathrm{Leb}$. Then this map satisfies all assumptions. 
\end{example}

\begin{example}[Finitely branched Gibbs-Markov maps]
    Let $T:[0,1]\to[0,1]$ be an interval map with finitely many branches with the Gibbs-Markov property (see the precise definition in \cite{HolNicTor12}), and suppose $T$ on each branch is $C^{1+\alpha}$. Let $\mu$ be the unique ergodic $T$-invariant measure absolutely continuous with respect to the Lebesgue measure, then $\mu$ also has the Gibbs property. It is well known that maps of this kind are hyperbolic recurrent \cite{HolTod25}, and satisfy \hyperref[cond: A1+]{(A1+)} \cite{Bow75}. Within each cylinder set, the invariant density is Lipschitz, hence as the boundary points of cylinder sets are countable, \hyperref[cond: LC]{(LC)} also holds.
\end{example}

\begin{example}[Gauss map]
    Let $G(x)=\frac1x-\lfloor \frac1x\rfloor$ be the Gauss map and $\mu_G$ the Gauss measure with density $\frac{d\mu_G}{d\mathrm{Leb}}(x)=\frac1{(\log 2)(1+x)}dx$. This system is hyperbolic recurrent, and its symbolic representation is a full-shift on natural numbers, the equilibrium with respect to the geometric potential $-\log |G'\circ \pi|$ is a Gibbs measure hence satisfies \hyperref[cond: A1+]{(A1+)}. Condition \hyperref[cond: LC]{(LC)} can be easily verified for every $x\in[0,1]$.
\end{example}

\section{Proof of Theorem~\ref{thm: convergence to Cox}}
The main tool to verify the convergence in question is by checking the convergence in the Laplace functional. In accurate terms, for $N_n=\sum_{j=1}^n \delta_{(j/n,Y_j-\log n)}$, we will check the following convergence
\begin{equation}\label{eqn: Laplace func of Cox}
    \begin{aligned}
        \lim_{n\to\infty}\E_\mu\left[e^{-\int g\,dN_n}\right]&= \E_\mu\left[\exp\!\left(-\iint (1-e^{-g(t,z)})\,2\rho(x)e^{-z}\,dzdt\right)\right]\\
        &=\int\exp\left(-\int (1-e^{-g(t,z)})\,2\rho(x)e^{-z}\,dzdt\right) \,\rho(x) dx
    \end{aligned}
\end{equation}
for every non‑negative continuous function \(g\) with compact support. Then, by \cite[Proposition 3.5]{Res87}, as the Laplace transform uniquely determines a measure, the desired convergence result holds. The major difficulty comes from the non-stationary nature of the random point measures $\{N_n\}_n$, which are defined over return distances, so the precise calculation of 
$\E_\mu[e^{-\int g N_n}]$ can be very complicated. Standard extreme value theory setup including Leadbetter's $\mathcal D$ and $\mathcal D'$ conditions (even the non-stationary version derived by \cite{FreFreVai17} as a dynamical adaptation of Husler's condition) are not applicable for this point measure.
Instead, we introduce an auxiliary point measure whose Laplace transform is easier to compute, and is useful in approximating that of $N_n$

\subsection{Spatial-Temporal partition}
The partition and related methods used in \cite{HolTod25} will be  handy for decorrelating the initial conditions. Let $\U_m=\{\Delta_k\}_k$ be the partition of $ \mathcal X$ using intervals $\Delta_k$(with disjoint interiors) of length $1/m$ and with centre $\zeta_k$. It is obvious that $\#\U_m= m$, and we choose 
$$m = n^{1+\beta'}  \text{ and set }\delta=|\Delta|/2.
$$ 

Let $r>0$. For each $\Delta\in\U_m$, set 
$$\Delta_0:=\{y\in\X:d(y,\Delta)\le r\},\hspace{4mm}\Delta^{\pm}:=\{y\in\X:d(y,\Delta)\le r\pm\delta\}.$$

\subsection{Approximation of Laplace functional by auxiliary point process}
In order to prove the convergence of the Laplace transform, we first need to approximate the behaviour of $N_n$ via some auxiliary point measures. Let $\U_m$ be the partition described above, for each $\Delta\in\U_m$, set 
$$ Y^\Delta_j(x):=-\log d(T^jx,\zeta),\text{ where $\zeta$ is the centre of $\Delta$.} 
$$

Let 
\begin{equation*}\label{eqn: def of N_n}
N_n^{\Delta}:=\sum_{j=1}^n\delta_{\left(j/n,Y_j^{\Delta}-\log n\right)}.
\end{equation*}
In comparison to the process $N_n$, which is of our main interest and related to the returns, $ N_n^{\Delta}$ corresponds to the hitting process to the $\Delta$-component. We then define our auxiliary point process: let $\widetilde N_n$ be a random point measure such that we
\begin{itemize}[noitemsep]
    \item draw $\Delta \in \U_m$ with probability $\mu(\Delta)$; 
    \item then draw $x\in\X$ according to $\mu$.
\end{itemize}
Note that $\widetilde N_n$ is still simple. 
Occasionally, abusing notation, we may write $\widetilde N_n=\sum_{\Delta\in\U_m}\mu(\Delta)N_n^\Delta$ (abusing in the sense that $\sum_{\Delta\in\U_m}\mu(\Delta)N_n^\Delta$ in its literal meaning is not a point measure, since its values may be non-integers).

Our first goal is to show that the Laplace transforms of these two point measures are asymptotically equal with an error term of order $1/\log n$. Namely, for each non-negative, continuous function $g:\R\to\R_{\ge0}$ with compact support, 
$$\mathcal L_n(g):=\E_\mu\left[e^{-\int g\,dN_n}\right]=\int e^{-\int g\,dN_n}d\mu,$$
$$\widetilde{\mathcal L_n}(g):=\E_\mu\left[e^{- \int g\,d\widetilde N_n}\right]=\int e^{- \int g\,d\widetilde N_n}d\mu=\sum_{\Delta\in \U_m}\mu(\Delta)\int e^{-\int g dN_n^\Delta} d\mu.$$
We aim to show $\lim_n |\mathcal L_n(g)-\widetilde{\mathcal L}_n(g)| \to0$. By the Mean Value Theorem, $|e^{-u}-e^{-v}|\le |u-v|$ for all $u,v\ge 0$, since $g$ is non-negative,  
$$
\E_\mu\left[\int gd\widetilde N_n-\int gdN_n\right]\le\mathcal L_n(g)-\widetilde{\mathcal L}_n(g)\le \E_\mu\left[\int gdN_n-\int gd\widetilde N_n\right].$$
Thus it suffices to prove the following proposition.

\begin{proposition}\label{prop: convergence of point process integral}
    Let $g:\R\to\R_{\ge0}$ be a non-negative continuous function with compact support and $N_n,\, \widetilde N_n$ defined as above. Then if $(\X,T,\mu)$ is hyperbolic recurrent,
\begin{equation}\label{eqn: auxiliary pp expectation converges to original pp}
   \E_\mu \left[ \left|\int g\, dN-\int g\, d\widetilde N_n \right|\right]\lesssim\frac1{\log n}.
\end{equation}
\end{proposition}

\begin{remark}\label{rmk: chance from continuous to C1}
When verifying the convergence of the Laplace functional for $N_n^\Delta$ for any $\Delta\in \U_m$, it is enough to restrict to $C^1$ non-negative test functions. This is because $C^1$ non-negative test functions are dense: for any $\eps>0$ and any $g$ a non-negative continuous functions with compact support $K=[a,b]\subset[-D,D]\subset (0,1]\times \R$, there exists $g_\eps\ge0$ that is $C^1$ and supported on $K$, such that $\|g-g_\eps\|_\infty<\eps$, hence
$$\left|\int gdN_n^\Delta-\int g_\eps dN_n^\Delta\right|\le \eps N_n^\Delta(K')\le \eps   N_n^\Delta((0,1]\times [-D,D])
$$
for some $0\le D<\infty$. Taking expectations on both sides, 
$$\E_\mu[N_n^\Delta((0,1]\times[-D,D])]\le n\mu(B(\zeta,e^{D}/n))\le 2e^D\|\rho\|_\infty<\infty,$$
where $\zeta$ is again the centre of $\Delta$. This inequality does not depend on $\Delta$ and $\eps$ was arbitrarily small, hence for any $g$ non-negative continuous function with compact support, for any $\Delta\in\U_m$, $\E_\mu\left[e^{-\int gdN_n^\Delta}\right]$ can be replaced by a Laplace functional of a $C^1$ function with arbitrarily small cost. 

The same argument applies to $N_n$ as well: by Lemma~\ref{lemma: estimates of return for longer time} below and assumption (A2),
$$\E_\mu[N_n((0,1]\times[-D,D])]\le \sum_{j=1}^n E_j(e^{D}/n)\lesssim n^{-\beta_1}+ne^{D}/n<\infty.
$$
\end{remark}

\begin{lemma}\label{lemma: estimates of return for longer time}
    For $(\X,T,\mu)$ a hyperbolic recurrent system, for $j>(\log n)^2$ and $r=r_n\asymp n^{-1}$,
    \begin{equation}\label{eqn: estimates of E_j(r_n)}
    \mu(E_j(r_n))\le 3\delta^{-1}e^{-(\log n)^2}+2\|\rho\|^2_\infty(r+\delta)\lesssim r\asymp n^{-1}.
    \end{equation}
\end{lemma}
\begin{proof}
For each $\Delta \in\U_m$, if $x\in E_j(r)\cap \Delta$, then $d(T^jx,\zeta)\le d(T^jx,x)+d(x,\zeta)\le r+\delta$. In other words, $x\in E_j(r)\cap \Delta$ implies that $\I_\Delta(x)\I_{\Delta^+}(T^jx)=1$. Hence, using decay of correlations (A1) and boundedness of the density $\rho$, as $\|\I_\Delta\|_{BV}\le 3$,
\begin{align*}
    \mu(E_j(r))\le&\sum_{\Delta\in\U_m}\int \I_\Delta\cdot \I_{\Delta^+}\circ T^j\,d\mu\le\sum_{\Delta\in\U_m}\left(\int \I_\Delta\,d\mu\int \I_{\Delta^+}\,d\mu+\|\I_\Delta\|_{BV}\|\I_{\Delta^+}\|_{L^1(\mu)}e^{-\eta j}\right)\\
    \le&3\delta^{-1}e^{-\eta j}+\sum_{\Delta\in\U_m}\mu(\Delta)\mu\left(\Delta^+\right)\le \delta^{-1}e^{-(\log n)^2}+2\|\rho\|^2_\infty(r+\delta)\qedhere.
\end{align*}
\end{proof}

\begin{proof}[Proof of Proposition~\ref{prop: convergence of point process integral}]
Let $g$ be a $C^1$ non-negative function with compact support. Given $\U_m$, we have
\begin{align*}
    &\int g\,dN_n(x)-\int g\,d\widetilde N_n(x)\\
    =&\sum_{\Delta\in \U_m}\I_\Delta(x)\sum_{j=1}^n g(j/n,Y_j(x)-\log n)-\sum_{\Delta\in\U_m} \sum_{j=1}^n \mu(\Delta)g\left(j/n,Y^\Delta_j(x)-\log n\right)\\
    =&\sum_{\Delta\in\U_m}\I_\Delta(x) \sum_{j=1}^n g(j/n,Y_j(x)-\log n)-g\left(j/n,Y^\Delta_j (x)-\log n\right)\\
    &+\sum_{\Delta\in\U_m}(\I_\Delta(x)-\mu(\Delta))\sum_{j=1}^n g\left(j/n,Y^\Delta_j(x)-\log n\right)\\
    =:&\text{Err}_{n,1}+\text{Err}_{n,2}.
\end{align*}
Note that $\E_\mu \left[\left|\int g\, dN-\int g\, d\widetilde N_n \right|\right]\le \E_\mu\left[\left|\text{Err}_{n,1}\right|\right]+ \E_\mu\left[\left|\text{Err}_{n,2}\right|\right]$. Because $g$ has compact support, many of the $g(j/n,\cdot)$ terms are $0$ unless $T^jx$ is either really close to the initial condition $x$, or the centre $\zeta$ of $\Delta$. This naturally allows us the to bound these error terms by restricting $x$ on $E_j(r_n)$ sets with $r_n\asymp n^{-1}$.
   
\subsubsection*{Step 1}
We first estimate $E_\mu[\text{Err}_{n,2}]$. Let 
$$
\phi^\Delta_{n,j}(\cdot):=g\left(j/n,-\log \left(d(\cdot,\zeta)\right)-\log n\right).
$$ 
Without loss of generality, suppose the support of $g$ in its $z$-coordinate is contained in $[a,b]$, and recall that $r_n(z):=\frac{e^{-z}}n$,
$$
\phi^{\Delta}_{n,j}(T^jx)\not\equiv0 \text{ when } d(T^jx,\zeta)\in\left[e^{-b}/n,e^{-a}/n\right]=:[r_n(b),r_n(a)].
$$
This gives 
$$
\|\phi_{n,j}^\Delta\|_{L^1(\mu)}\le \|g\|_\infty\cdot \mu\left(\supp(\phi_{n,j}^\Delta)\right)\le\|g\|_\infty \cdot \mu\left(B(\zeta,r_n(a)\right). 
$$
Since the density $\rho=\frac{d\mu}{dLeb}\in \mathcal{BV}$ is bounded, we can conclude that $\|\phi_{n,j}^\Delta\|_{L^1(\mu)}\lesssim n^{-1}$
Also, in order that $\I_\Delta(x)\phi_{n,j}^\Delta(x)\not\equiv0$, we must have $d(T^jx,x)\le e^{-a}/n+\delta$, that is, $x\in E_{j}(r_n(a)+\delta)$. 

By assumptions (A1) and (A2),
\begin{align*}
    &\E_\mu\left[\sum_{\Delta\in\U_m}\int\I_\Delta(x)\sum_{j=1}^n\phi_{n,j}^\Delta(T^jx)\,d\mu(x)\right]\\
    &=\sum_{\Delta\in\U_m}\int\I_\Delta(x)\sum_{j=1}^{(\log n)^2}\phi_{n,j}^\Delta(T^jx)d\mu(x)+
    \sum_{\Delta\in\U_m}\int \I_\Delta(x)\sum_{j>(\log n)^2}^n\phi_{n,j}^\Delta(T^jx)d\mu(x)\\
    \le& C_0(\log n)^2\|g\|_\infty (r_n(a)+\delta)^{\beta_0}+\sum_{\Delta\in\U_m}\sum_{j>(\log n)^2}^n\int \I_\Delta d\mu\int\phi_{n,j}^\Delta d\mu+ C_T\|\I_\Delta\|_{BV}\|\phi_{n,j}^\Delta\|_{L^1(\mu)}e^{-\eta(\log n)^2}\\
    \le&C_0(\log n)^2\|g\|_\infty (r_n(a)+\delta)^{\beta_0}+n^{2+\beta'}C_T 3\|g\|_\infty n^{-\eta \log n}+\sum_{\Delta\in\U_m}\mu(\Delta)\sum_{j>(\log n)^2}^n\|\phi_{n,j}^\Delta\|_{L^1(\mu)} 
\end{align*}
Then, 
\begin{align*}
\E_\mu[    \text{Err}_{n,2}]\le&C_0(\log n)^2\|g\|_\infty (r_n(a)+\delta)^{\beta_0}+ n^{2+\beta'}C_T 3\|g\|_\infty n^{-\eta \log n}-\sum_{\Delta\in\U_m}\sum_{j=1}^{(\log n)^2}\mu(\Delta)\|\phi_{n,j}^\Delta\|_{L^1(\mu)}\\
    \lesssim&(\log n)^2 n^{-\beta_0}+ n^{2+\beta'}n^{-\eta \log n}-(\log n)^2n^{-1}\lesssim \frac1{\log n}.
\end{align*}
Similarly, by completely ignoring the first $(\log n)^2$ iterates in $\sum \I_\Delta\sum_{j=1}^n\phi_{n,j}$, we bound the expectation from below by
\begin{align*}
    \E_\mu[\text{Err}_{n,2}]\ge&-\sum_{\Delta\in\U_m}\mu(\Delta)\sum_{j=1}^{(\log n)^2}\|\phi_{n,j}^\Delta\|_{L^1(\mu)} +n^{2+\beta'}C_T 3\|g\|_\infty n^{-\eta \log n}\\
    \gtrsim&-\sum_{\Delta\in\U_m}\mu(\Delta)(\log n)^2n^{-1}-n^{2+\beta'} n^{-\eta \log n} \gtrsim -(\log n)^{-1}.
\end{align*}
We therefore conclude that $\lim_{n\to\infty}\left|\E_\mu[\text{Err}_{n,2}]\right|=0$ with an error term uniformly asymptotically bounded by $\frac1{\log n}$.

\subsubsection*{Step 2}
Next we treat $\text{Err}_{n,1}$. 
    
    Fix $\Delta\in\U_m$, recall by definition, for all $x\in\Delta$, $d(x,\zeta)\le \delta$, hence by triangular inequality, for all $j$,
    \begin{equation}\label{eqn: diff in g value}\left|d\left(x,T^jx\right)-d\left(T^jx,\zeta\right)\right|\le d(x,\zeta)\le \delta=n^{-(1+\beta')}.
    \end{equation}
    If $d(T^jx,\zeta)> r_n(a)+\delta$ then $d(T^jx,x)>r_n(a)$ hence both $g(j/n,Y_j(x)-\log n)$ and $g\left(j/n,Y^\Delta_j -\log n\right)$ are automatically zero. Similarly if $d(T^jx,\zeta)<r_n(b)-\delta$ then both $g$ terms are $0$. So neither of $g(j/n,Y_j-\log n)$ and $g(j/n,Y^\Delta_j -\log n)$ is automatically zero, for all $j=1,\dots, n$, only when $d(T^jx,x)\in[r_n(b)+\delta,r_n(a)-\delta]$. In this case, $\frac{|d(T^jx,x)-d(T^jx,\zeta)|}{d(T^jx,\zeta)}\le \frac{\delta}{e^{-b}n^{-1}/2}\le 2e^{b}n^{-\beta'}$. As for all $z$ small, by Taylor expansion at $0$ we have the rough upper bound $|\log (1-z)|\le 2|z|$,
    \begin{equation}\label{eqn: g value diff}
        \left|Y_j(x)-\log n-(Y^\Delta_j (x)-\log n)\right|=\left|\log\left(1-\frac{d(T^jx,x)-d(T^jx,\zeta)}{d(T^jx,x)}\right)\right| \le 4e^bn^{-\beta'}.
    \end{equation}
    
   Define 
   \begin{equation}\label{eqn: def of omega function}
   \omega_\ell(g):=\sup\{|g(t,u)-g(t,v)|: t,u,v\in\supp(g), \ |u-v|\le \ell\},
   \end{equation}
    as $g$ is $C^1$ and has compact support in $z$ coordinate, $\omega(g)_\ell\lesssim \ell$. Also define the sets
    $$
    \Delta_j^{\text{bulk}}:=\left\{x\in\Delta:d(T^jx,\zeta)\in\left[\frac{e^{-b}}n+\delta,\frac{e^{-a}}n-\delta\right]\right\}
    $$
    $$
    \Delta_j^{\text{shell}}:=\left\{x\in\Delta:d(T^jx,\zeta)\in\left[\frac{e^{-b}}n-2\delta,\frac{e^{-b}}n+\delta\right]\cup\left[\frac{e^{-a}}n-\delta,\frac{e^{-a}}n+2\delta\right]\right\}
    $$
Note that 
\begin{itemize}
    \item %$\bigcup_{\Delta\in\U_m}\Delta_j^{\text{bulk}}\subset E_j(r_n(a))$, 
    $\Delta_j^{\text{bulk}}\subset E_j(r_n(a))\cap\Delta$,
    $\bigcup_{\Delta\in\U_m}\Delta_j^{\text{shell}}\subset E_j(r_n(a)+3\delta)$;
    \item $\Delta_j^{\text{bulk}}\cap\Delta_j^{\text{shell}}=\emptyset$;
    \item $g\left|_{\Delta\setminus\left(\Delta_j^{\text{bulk}}\cup\Delta_j^{\text{shell}}\right)}\right.\equiv 0$;
    \item For any $r_1<r_2$, let $A(\zeta,r_1,r_2):=B(\zeta,r_2)\setminus B(\zeta,r_1)$. Then
    \begin{equation}\I_{\Delta_j^{\text{shell}}}=\I_\Delta \cdot \left(\I_{A(\zeta,r_n(a)-\delta,r_n(a)+2\delta)}\cup \I_{A(\zeta,r_n(b)-\delta,r_n(b)+\delta)}\right)\circ T^j
    \end{equation}
    
\end{itemize}
Let us reduce the cumbersome notation above by defining $$
A'(\zeta,a,b,\delta):=A(\zeta,r_n(a)-\delta,r_n(a)+2\delta)\cup A(\zeta,r_n(b)-\delta,r_n(b)+\delta).
$$
By boundedness of the density $\rho$ 
\begin{equation}\label{eqn: shell measure}
\mu\left(A'(\zeta,a,b,\delta)\right)\\
=\mu\left(B(\zeta,r_n(a)+2\delta)\setminus B(\zeta,r_n(a)-\delta)\right)+\mu\left(B(\zeta,r_n(b)+2\delta)\setminus B(\zeta,r_n(b)-\delta)\right)
\lesssim\delta.
\end{equation}

Using both (A1) (decay of correlations) and (A2) (estimate of short returns), and the invariance of $\mu$ together with \eqref{eqn: shell measure},
\begin{align*}
\sum_{\Delta\in\U_m}\sum_{j=1}^n \int\ \I_{\Delta_j^{\text{shell}}}\,d\mu=&\sum_{\Delta\in\U_m}\sum_{j=1}^{(\log n)^2} \int
\ \I_{\Delta_j^{\text{shell}}}d\mu+\sum_{\Delta\in\U_m}\sum_{j>(\log n)^2}^n \int \I_{\Delta_j^{\text{shell}}}\,d\mu\\
\le& \widetilde E_{(\log n)^2}(r_n(a)+3\delta)+\sum_{\Delta\in\U_m}\sum_{j>(\log n)^2}^n\left(\mu(\Delta)\mu\left(A'(\zeta,a,b,\delta)\right)+3e^{-\eta (\log n)^2}\right)\\
\lesssim& n^{-\beta_1}+n\delta^{-1}(\delta^2+n^{-\eta (\log n)^2})\lesssim n^{-\beta_1}+n^{-\beta'}.
\end{align*}

Therefore,
\begin{align*}
&\sum_{\Delta\in\U_m}\sum_{j=1}^n\int_{\Delta_j^{\text{shell}}}\left|g(j/n,Y_j(x)-\log n)-g\left(j/n,Y^\Delta_j(x)-\log n\right)\right|   d\mu(x)
\\
&\le 2\|g\|_\infty\sum_{\Delta\in\U_m}\sum_{j=1}^n \int\ \I_{\Delta_j^{\text{shell}}}\,d\mu\lesssim n^{-\beta_1}+n^{-\beta'}
\end{align*}

For any fixed $j$ and different $\Delta$ and $\Delta'\in\U_m$, $\Delta_j^{\text{bulk}}\cap\Delta_j^{'\text{bulk}}=\emptyset$, and
\begin{align*}
%\\
&\int\sum_{\Delta\in\U_m}\sum_{j=1}^n\I_{\Delta_j^{\text{bulk}}}(x)\left|g(j/n,Y_j(x)-\log n)-g\left(j/n,Y^\Delta_j(x)-\log n\right)\right|   d\mu(x)
\\
\le&\sum_{j=1}^n\int_{\bigcup_{\Delta\in\U_m}\Delta_j^{\text{bulk}}}\left|g(j/n,Y_j(x)-\log n)-g\left(j/n,Y^\Delta_j(x)-\log n\right)\right|
\\
\le &\omega_{4e^bn^{-\beta'}}(g)\sum_{j=1}^n \mu\left(\bigcup_{\Delta\in\U_m}\Delta_j^{\text{bulk}}\right)\le\omega_{4e^bn^{-\beta'}}(g)\sum_{j=1}^n\mu\left(E_j(r_n(a))\right)
\end{align*}

By (A2) and \eqref{eqn: estimates of E_j(r_n)}, 
\begin{align*}
    \sum_{j=1}^n\mu(E_j(r_n(a)))\le C_1 (r_n(a))^{\beta_1}+n\left(3\delta^{-1}e^{-(\log n)^2}+2\|\rho\|^2_\infty(r_n(a)+\delta)\right)<\infty.
\end{align*}

Lastly, combining all the error estimates and use the fact that $g\in C^1$ so that $\omega_{4e^bn^{-\beta'}}(g)\lesssim n^{-\beta'}$,
\begin{align*}
\E_\mu[|\text{Err}_{n,1}|]\le&\int\sum_{\Delta\in\U_m}\sum_{j=1}^n\I_\Delta(x)\left|g(j/n,Y_j(x)-\log n)-g\left(j/n,Y^\Delta_j(x)-\log n\right)\right|   d\mu(x)\\
\lesssim&n^{-\beta_1}+n^{-\beta'}+\omega_{4e^bn^{-\beta'}}(g)\lesssim n^{-\beta_1}+n^{-\beta'}+n^{-\beta'}\lesssim \frac1{\log n}.\qedhere
\end{align*}
\end{proof}

\subsection{Limit of Laplace functional for the auxiliary process.}
It remains to show that the Laplace function of $\widetilde N_n$ does converge to the desired limit with a similar rate, \emph{i.e.}\ in the form of \eqref{eqn: Laplace func of Cox}. Especially, we will show that the Laplace functional of the sum over $N_n^\Delta$ for converges to that of a Cox process. Almost all quantities in the following section are defined with respect to some point $\zeta\in\X$. When an object has $\Delta$ as a superscript, it means the relevant point $\zeta=\zeta_\Delta$ is the centre of a partition set $\Delta\in \U_m$. Otherwise, it is an arbitrary $\zeta\in \X$.

\begin{proposition}\label{prop: uniform convergence for localised point processes}
Let $g$ be an arbitrary non-negative $C^1$ function with compact support. Define
$$
\Psi^g_\infty(\zeta):=\exp\left(-\iint \left(1-e^{-g(t,z)}\right)2\rho(\zeta)e^{-z}\,dz\,dt\right).
$$
Under the conditions of Theorem~\ref{main thm: convergence of PP}, 
\begin{equation}\label{eqn: uniform local convergence}
\sum_{\Delta\in\U_m}\mu(\Delta)\,
\left|\E_\mu\left[e^{-\int g\,dN^{\Delta}_n}\right]-\Psi_\infty^g(\zeta_\Delta)\right|
\lesssim \frac1{\log n},
\end{equation}
where $\zeta_\Delta$ is the centre of $\Delta\in\U_m$.
\end{proposition}

We will again prove this proposition in several steps. Fix $g$ a $C^1 $ non-negative test function such that $\supp (g)\subset (0,1]\times [a,b]$. First, approximate $g$ by simple functions. Fix $K\in\mathbb N$ and let
$$
0=t_0<t_1<\cdots<t_K=1
$$
be a partition of equal length over the support of $g$ in its $t$ coordinate. Similarly set
$$
a=z_0<z_1<\cdots<z_L=b,
$$
which is a partition of the $z$-support and without loss of generality is assumed to be contained in $[a,b]$. Set
\begin{equation}\label{eqn: def I_k}
I_k=\left\{j\in\{1,\ldots,n\}:\frac{j}{n}\in[t_{k-1},t_k)\right\}.
\end{equation}
With the above partition, we prove the following error bound.

\begin{lemma}\label{lemma: laplace func for step function}
For $\widetilde g=\widetilde g_L$ defined as above, for every $\Delta\in\U_m$,
    \begin{equation}\label{eqn: step function Laplace error}
        \left|\E_\mu\left[e^{-\int \widetilde g d N^{\Delta}_n}\right]-\E_\mu\left[e^{-\int  g d N^{\Delta}_n}\right]\right|\lesssim \max\left\{\frac1L,\frac1K\right\}.
    \end{equation}
\end{lemma}
\begin{proof}
By $C^1$ and compact support, $\|D(g)\|_\infty<\infty$, so using the Mean Value Theorem,
$$\sup_{t,z}|g(t,z)-\widetilde g(t,z)|\le \|D(g)\|_\infty\max\left\{\frac1L,\frac1K\right\}.
$$ 
Then up to a uniform constant depending on $g$,
\begin{align*}
\left|\E_\mu\left[e^{-\int \widetilde g d N^{\Delta}_n}\right]-\E_\mu\left[e^{-\int  g d N^{\Delta}_n}\right]\right|&\le\left|\E_\mu\left[\int g d N_n^\Delta\right]-\E_\mu\left[\int \widetilde g dN_n^\Delta\right]\right|\lesssim \max\left\{\frac1L,\frac1K\right\} E_\mu\left[N_n^\Delta([a,b])\right]\\
&\lesssim \max\left\{\frac1L,\frac1K\right\}n\mu(B(\zeta,e^{-a}/n))\lesssim\max\left\{\frac1L,\frac1K\right\}.\qedhere
\end{align*}
\end{proof}

We choose $K=L=\log n$, so the error between their Laplace functionals is of order $1/\log n$, then compute $\E_\mu\left[e^{-\int \widetilde g dN_n^\Delta}\right]$. Define the annuli 
$$
A_l(\zeta)
=
B(\zeta,r_n(z_{l-1}))
\setminus
B(\zeta,r_n(z_l)),
\qquad
r_n(z)=\frac{e^{-z}}n.
$$
hence the counting variable
$$
M_{n,k,l}=\sum_{j\in I_k}\mathbf 1_{A_l(\zeta)}(T^jx),
\text{ for each }
1\le k\le K,\text{ and } 1\le l\le L,
$$
where $I_k=\left\{j:\,j/n\in[t_{k-1},t_k)\right\}$. Then, 
\begin{align*}
    \int \widetilde g dN_n^\Delta(x)&=\sum_{j=1}^n \widetilde g\left(j/n,Y^\Delta_j (x)-\log n\right)\\
    &=\sum_{l=1}^L\sum_{k=1}^K \sum_{j=1}^n g(t_{k-1},z_{l-1})\I_{[t_{k-1},t_k)}(j/n)\I_{[z_{l-1},z_l)}\left(Y^\Delta_j(x)-\log n\right)\\
    &=\sum_{l=1}^L\sum_{k=1}^K g(t_{k-1},z_{l-1})M^\Delta_{n,k,l}(x)
\end{align*}

Hence
$\E_\mu\left[e^{-\int \widetilde g dN_n^\Delta}\right]
$ 
is a function of the random vectors
$\left(M_{n,k,1}^\Delta,\dots,M_{n,k,L}^\Delta\right)$.

Define the total variation metric $d_{\mathrm{TV}}$: for two non-negative integer valued random variables, $X,Y$, 
$$d_{\mathrm{\mathrm{TV}}}(\L(X),\L(Y)):=\frac12\sum_{k=0} |\P(X=k)-\P(Y=k)|
$$
where $\L$ denotes the law (distribution) of a random variable. In particular, if the two random variables are defined over the same probability space, then $d_{\mathrm{\mathrm{TV}}}(\L(X),\L(Y))\le \P(X\neq Y)$.

Below we also recall a multivariate Poisson approximation theorem from \cite[Theorem 1]{Bar88}, developed from \cite{LCam60}.
\begin{theorem}[Multivariate Le Cam's Poisson Approximation theorem]\label{thm: Le Cam}
    Let $X_1,\dots,X_n$ be independent Bernoulli random $d$-vectors such that $\P(X_j=\epsilon_i)=p_j^{(i)}$, where $\epsilon_i$ represents the $i$'th coordinate unit vector. Let $p_j=\sum_{i=1}^d p_j^{(i)}$,
     then
    $$d_{\mathrm{\mathrm{TV}}}\left(\mathcal L\left(\sum_{i=1}^n X_i\right),\mathcal L\left(Z_1,\dots,Z_d\right)\right)\le\sum_{j=1}^n p_j^2.
    $$
    where $Z_1,\dots, Z_d$ are independent Poisson vectors, such that 
    $$Z_i\sim \text{Poi}\left(\sum_{j=1}^n p_j^{(i)}\right).
    $$
\end{theorem}
The main idea towards proving Proposition~\ref{prop: uniform convergence for localised point processes} is via coupling the locally defined point measures, $N_n^\Delta$, with Poisson random variables via the total variation metric. More precisely, we will slice the time window into blocks of size $(\log n)^3$, separated by some gaps, then apply the decorrelation argument. The stronger type of mixing (A1+) is needed as we need to deal with decorrelate the events of hitting to annuli sets between each blocks, and (A1) is not sufficient as some of the indicator functions after composing with the dynamic $T$ can have large BV norms. 

\begin{lemma}[Space-time Total-Variation approximation] Under the conditions of Theorem~\ref{main thm: convergence of PP}, 
\label{lemma: TV approximation}
$$
\sum_{\Delta\in\U_m}\mu(\Delta)\,
d_{\mathrm{\mathrm{TV}}}\left(
\L\left((M_{n,k,l}^{\Delta})_{1\le k\le K,\,1\le l\le L}\right),
\L\left((Z^\Delta_{k,l})_{1\le k\le K,\,1\le l\le L}\right)
\right)
\lesssim \frac1{\log n},
$$
where $\zeta_\Delta$ is the centre of $\Delta\in\U_m$, and the random variables $Z_{k,l}$ are independent Poisson random variables with parameters
$$
\lambda_{k,l}=2\rho(\zeta_\Delta)(t_k-t_{k-1})\int^{z_l}_{z_{l-1}}e^{-z}\,dz.
$$
\end{lemma}

\begin{proof}

Proof of this lemma splits into four steps, with similar ideas used in \cite{ChaCol13}, namely coupling the relevant random variables with independent Bernoulli-like random variables.

\textbf{Step 1: Breaking into gaps.}

We write $b_n=(\log n)^3,$ $\ell_n=(\log n)^2$ and $c_n=b_n+\ell_n,$
and set $k_n=\left\lfloor\frac{n}{c_n}\right\rfloor.$ We then partition ${1,\ldots,n}$ into blocks
$$
B_i=\{(i-1)c_n+1,\ldots,(i-1)c_n+b_n\},\qquad
1\le i\le k_n,
$$
with gaps of length $\ell_n$ between consecutive blocks.
For $1\le k\le K$ and $1\le l\le L$. Let $\zeta\in \X$, define the block counts
$$
S_{i,k,l}=\sum_{j\in B_i\cap I_k}\mathbf 1_{A_l(\zeta)}(T^jx).
$$
Thus we can write
$$
M_{n,k,l}=\sum_{i=1}^{k_n}S_{i,k,l}+R_{n,k,l},
$$
where
$R_{n,k,l}=\sum_{j\notin\cup_{i=1}^{k_n}B_i}
\I_{I_k}(j)\I_{A_l(\zeta)}(T^jx)$ and $I_k$ as in \eqref{eqn: def I_k}.

Let $R_n^{\text{gap}}:=\sum_{k,l}R_{n,k,l}$, and summing over $k,l$, we have, uniformly in $\zeta$,
$$
\sum_{k,l}\E_\mu[R_{n,k,l}]\le \sum_{j\not\in \cup_{i=1}^{k_n}B_i}\E_\mu[\I_{B(\zeta,r_n(z_0))}\circ T^j]\lesssim k_n\ell_nn^{-1}\lesssim \frac1{\log n}.
$$
Hence, summing over all centres $\zeta_\Delta$,
\begin{equation}\label{eqn: gap removal}
\sum_{\Delta\in\U_m}\mu(\Delta)\,
d_{\mathrm{\mathrm{TV}}}
\left(\L\left((M_{n,k,l}^{\Delta})_{k,l}\right),
\L\left((S^\Delta_{i,k,l})_{i,k,l}\right)\right)
\lesssim \frac{1}{\log n}.
\end{equation}

\medskip
\noindent
\textbf{Step 2: At most one hit within blocks.}

Fix $\zeta\in\X$, suppressing the notation dependence on $\zeta$, we set $N_i(\zeta)=\sum_{k=1}^K\sum_{l=1}^L S_{i,k,l}.$ Since the annuli sets $A_l(\zeta)$ are disjoint and $\bigcup_{l}A_l(\zeta)\subset B(\zeta,r_n(z_0))$, 
$$\{N_i(\zeta)\ge2\}\subset\{\exists\,j,j':\, T^jx, T^{j'}x\in B(\zeta,r_n(z_0))\}.
$$
By Markov's inequality
\begin{align*}
\mathbb P(N_i(\zeta)\ge2)
&\le\frac12
\sum_{\substack{j,j'\in B_i:\, j\ne j'}}
\mu\left(
B(\zeta,r_n(z_0))
\cap
T^{-(j'-j)}B(\zeta,r_n(z_0))
\right).
\end{align*}

We will bound the averaged probability of having more than one hits to any of these annuli sets within each block $B_i$, \emph{i.e.}\
$$
\sum_{\Delta\in\U_m}\mu(\Delta)\sum_{i=1}^{k_n}\mathbb P(N_i^\Delta\ge2)\le \sum_{\Delta\in\U_m}\mu(\Delta)\sum_{i=1}^{k_n}\sum_{j,j'\in B_i:j<j'}\E_\mu\left[\I_{B\left(\zeta_\Delta,r_n(z_0)\right)}\circ T^j\cdot \I_{B\left(\zeta_\Delta,r_n(z_0)\right)}\circ T^{j'}\right].
$$
We split the sum according to whether $j'-j\le\ell_n$ or $j'-j>\ell_n.$

{\bf Short returns}: $j'-j\le\ell_n$.

For fixed $\zeta$ ,
$$
\sum_{\substack{j,j'\in B_i\\0<j'-j\le\ell_n}}\mu\left(
B(\zeta_\Delta,r_n(z_0))
\cap
T^{-(j'-j)}B(\zeta,r_n(z_0))
\right)
\le b_n\sum_{k=1}^{\ell_n}\mu\left(
B(\zeta,r_n(z_0))\cap T^{-k}B(\zeta,r_n(z_0))
\right).
$$
Summing over all $k_n$ blocks of digits and averaging over $\Delta\in\U_m$,
$$
\begin{aligned}
&\sum_{\Delta\in\U_m}\mu(\Delta)\sum_{i=1}^{k_n}\sum_{\substack{j,j'\in B_i\\0<j'-j\le\ell_n}}\mu\left(
B(\zeta_\Delta,r_n(z_0))\cap T^{-(j'-j)}B(\zeta_\Delta,r_n(z_0))
\right)\\
&\le k_nb_n\sum_{\Delta\in\U_m}\mu(\Delta)\sum_{j=1}^{\ell_n}\mu\left(
B(\zeta_\Delta,r_n(z_0))\cap T^{-j}B(\zeta_\Delta,r_n(z_0))
\right).
\end{aligned}
$$
The function $\zeta\mapsto\sum_{j=1}^{\ell_n}\mu\left(B(\zeta,r_n(z_0))\cap T^{-j}B(\zeta,r_n(z_0))\right)$ is Lipschitz with constant $\lesssim\ell_n\|\rho\|_\infty$. Hence by mean value theorem, the Riemann sum over the partition $\U_m$ of diameter $\delta=n^{-(1+\beta')}$ satisfies
\begin{align*}
&\sum_{\Delta\in\U_m}\mu(\Delta)\sum_{k=1}^{\ell_n}\mu\left(B(\zeta_\Delta,r_n(z_0))\cap T^{-k}B(\zeta_\Delta,r_n(z_0))\right)\\
&\le\int\sum_{k=1}^{\ell_n}\mu\left(B(\zeta,r_n(z_0))\cap T^{-k}B(\zeta,r_n(z_0))\right)\,d\mu(\zeta)
+\mathcal O\left((\log n)^2n^{-(1+\beta')}\right). 
\end{align*}
By Lemma~\ref{lemma: A2 implies A2+},
$$
\sum_{\Delta\in\U_m}\mu(\Delta)\sum_{k=1}^{\ell_n}\mu\left(B(\zeta_\Delta,r_n(z_0))\cap T^{-k}B(\zeta_\Delta,r_n(z_0))\right)\lesssim\frac{1}{n\log n}.
$$
Multiplying by $k_nb_n\lesssim n$,
\begin{equation}\label{eqn: short-lag-averaged}
\sum_{\Delta\in\U_m}\mu(\Delta)\sum_{i=1}^{k_n}\sum_{\substack{j,j'\in B_i\\0<j'-j\le\ell_n}}\mu\left(
B(\zeta_\Delta,r_n(z_0))\cap T^{-(j'-j)}B(\zeta_\Delta,r_n(z_0))
\right)
\lesssim\frac{1}{\log n}.
\end{equation}

{\bf Long returns}: $j'-j>\ell_n$.

For any fixed $\zeta\in \X$, by \hyperref[cond: A1]{(A1)} decay of correlations,
\begin{align*}
&\mu\left(
B(\zeta,r_n(z_0))
\cap
T^{-(j'-j)}B(\zeta,r_n(z_0))
\right)\\
&\le
\mu(B(\zeta,r_n(z_0)))^2
+C_T
\|\I_{B(\zeta,r_n(z_0))}\|_{BV}\|\I_{B(\zeta,r_n(z_0))}\|_{L^1(\mu)}
e^{-\eta(j'-j)}.
\end{align*}
Since $\mu(B(\zeta,r_n(z_0)))\lesssim n^{-1},$ and
$\|\I_{B(\zeta,r_n(z_0))}\|_{BV}\lesssim1,$
this is uniformly (in $\zeta$) bounded by $n^{-2}+n^{-1}e^{-\eta(j'-j)}$. Summing over at most $b_n^2$ pairs within each $B_i$ block and averaging over $\Delta\in\U_m$,
\begin{equation}\label{eqn: long-lag-averaged}
\begin{aligned}
&\sum_{\Delta\in\U_m}\mu(\Delta)\sum_{i=1}^{k_n}\sum_{\substack{j<j'\in B_i: j'-j>\ell_n}}\mu\left(
B(\zeta_\Delta,r_n(z_0))\cap T^{-(j'-j)}B(\zeta_\Delta,r_n(z_0))
\right)\\
&\lesssim k_nb_n^2\left(n^{-2}+n^{-1}e^{-\eta(\log n)^2}\right)
\lesssim (\log n)^3n^{-1}.
\end{aligned}
\end{equation}

Combining \eqref{eqn: short-lag-averaged} and \eqref{eqn: long-lag-averaged},
\begin{equation}\label{eqn: step2-averaged}
\sum_{\Delta\in\U_m}\mu(\Delta)\sum_{i=1}^{k_n}\mathbb P(N^\Delta_i\ge2)\lesssim\frac{1}{\log n}+\frac{(\log n)^3}{n}\lesssim\frac{1}{\log n}.
\end{equation}

Now define the random $k\times l$ vector $S^\Delta_i:=(S^\Delta_{i,k,l})_{k,l},$
and set $\beta^\Delta_i=S_i^\Delta\I_{\{N_i(\zeta_\Delta)\le1\}}.$ Then $\beta^\Delta_i$ takes values in
$$
\{0\}\cup
\{e_{k,l}:1\le k\le K,\;1\le l\le L\},
$$
where $e_{k,l}$ denotes the corresponding coordinate matrix.
By the coupling characterisation of total variation,
\begin{equation}\label{eqn: TV distance to non-indpt Bernoulli}
\sum_{\Delta\in\U_m}\mu(\Delta)\,
d_{{\mathrm{TV}}}
\left(
\L((S^\Delta_i)_i),\L((\beta^\Delta_i)_i)\right)
\lesssim \frac1{\log n}.
\end{equation}

\textbf{Step 3: Decorrelation of the block variables.}
Let $\beta_i$ denote the random variable with the same construction as above, only substituting an arbitrary $\zeta\in \X$ for a centre of a partition set $\zeta_\Delta\in \Delta$. We will show that, uniformly in $\zeta$,
\begin{equation}\label{eq: block decorrelation}
d_{\mathrm{\mathrm{TV}}}\left(
  \L(\beta_1,\ldots,\beta_{k_n}),
  \bigotimes_{i=1}^{k_n}\L(\beta_i)
\right)
\lesssim (\log n)\,n^{-2}.
\end{equation}
The decay of correlation condition (A1) is not quite enough to prove \eqref{eq: block decorrelation} because the indicator function of $\I\{\beta_i=s\}$ may have very large BV norm. The main tool to bypass such difficulty is symbolic coding and using cylinders to approximate the annulus sets $A_l(\zeta)$ by finite-depth cylinder unions. Then, we can apply the mixing property which takes account of all events in the relative sigma-algebras.

\textbf{Step 3(a): Cylinder approximation of the annuli.}

The following technique is inspired by \cite{BarJurKlo23}. Let $\mathcal I$ be the Markov partition of $T$ as described at the beginning of \S~\ref{sec: main results}, and let $\mathcal I^{(m)}$ denote the partition into cylinders of depth $m$. By the uniform expansion assumption there exist constants $C_2>0$ and $\upsilon>1$ such that
$$
\sup_{I\in\mathcal I^{(m)}}\operatorname{diam}(I)\le C_2\upsilon^{-m},\qquad m\ge1,
$$
and the invariant measure $\mu$ has a bounded density. For $\gamma>0$, consider the stopping-cylinder partition
$$
\mathcal W(\gamma):=\left\{I\in\mathcal I^{(m)}:\operatorname{diam}(I)\le\gamma<\operatorname{diam}(I_{-})\right\},
$$
where $I_-$ denotes the corresponding parent cylinder obtained by deleting the last symbol. It is easy to see that $\mathcal W(\gamma)$ is a partition of the repeller. Moreover, there exists $m(\gamma)=\mathcal O(\log(1/\gamma))$ such that the symbolic representation of every cylinder set in $\mathcal W(\gamma)$ has length at most $m(\gamma)$.

For each annulus set $A_l(\zeta)$, define its inner approximation by
$$
A_{l,n}(\zeta):=\bigcup_{\substack{I\in\mathcal W(\gamma_n)\\ I\subset A_l(\zeta)}} I,
\qquad \gamma_n:=n^{-3}.
$$
Although this union may be countably infinite, any cylinder set appearing in $A_{l,n}(\zeta)$ is determined by the first $m_n$ symbols in $\Sigma$, where
$$
m_n:=m(\gamma_n)\lesssim \log n.
$$
Since each $A_l(\zeta)$ has a uniformly bounded number of boundary points, every cylinder contributing to $A_l(\zeta)\setminus A_{l,n}(\zeta)$ must meet one of those boundary points. As each stopping cylinder has diameter at most $\gamma_n$, boundedness of the density $\rho$ yields
$$
\mu\left(A_l(\zeta)\triangle A_{l,n}(\zeta)\right)\lesssim\gamma_n,
$$
uniformly in $\zeta$ and $l$. Write $\beta_i^{(n)}$ for the block vector obtained from $\beta_i$ by replacing every indicator $\I_{A_l(\zeta)}\circ T^j$ by $\I_{A_{l,n}(\zeta)}\circ T^j$. Since $L=\log n$, invariance of $\mu$ and the union bound give
\begin{equation}\label{eqn: cylinder approximation}
\begin{aligned}
&\mu\left(\exists\,j\le n,\ l\le L:\,
\I_{A_l(\zeta)}(T^j x)
\ne \I_{A_{l,n}(\zeta)}(T^j x)\right)\\
&\le\sum_{j=1}^n\sum_{l=1}^L
\mu\left(A_l(\zeta)\triangle A_{l,n}(\zeta)\right)
\le Cn\gamma_n L
\lesssim (\log n)n^{-2}.
\end{aligned}
\end{equation}

\textbf{Step 3(b): Decorrelation of the approximating block variables.}
The elements in each approximating target set are determined by at most $m_n$ symbols. Write $\mathcal G_p^q$ as the $\sigma$-algebra generated by the set $\sigma^p(\mathcal C_q)$, where $\mathcal C_q$ is the set of $q$-cylinders. The approximating block vector $\beta_i^{(n)}$ takes values in the finite set $\{0\}\cup\{e_{k,l}:1\le k\le K,\,1\le l\le L\}$, and is measurable with respect to
$$
\mathcal F_i:=\mathcal G_{s_i+1}^{\,s_i+b_n+m_n},\qquad s_i=(i-1)(b_n+\ell_n).
$$
Consecutive $\sigma$-algebras are separated by
$$
w_n:=\ell_n-m_n-1\asymp(\log n)^2.
$$
Let
$$
\L_i:=\L(\beta_1^{(n)},\dots,\beta_i^{(n)})\otimes\bigotimes_{j=i+1}^{k_n}\L(\beta_j^{(n)}),\qquad 0\le i\le k_n.
$$
Then $\L_0=\bigotimes_{j=1}^{k_n}\L(\beta_j^{(n)})$ is the product of the marginals and $\L_{k_n}=\L(\beta_1^{(n)},\dots,\beta_{k_n}^{(n)})$ is the joint law.

Fix $1\le i\le k_n$ and a bounded test function $h$ with $\|h\|_\infty\le 1$. Define
$$
\Phi(\beta_1^{(n)},\dots,\beta_i^{(n)})
:=\mathbb E_{\bigotimes_{j>i}\L(\beta_j^{(n)})}
\left[h(\beta_1^{(n)},\dots,\beta_i^{(n)};\beta_{i+1}^{(n)},\dots,\beta_{k_n}^{(n)})\right].
$$
Then
$$
\mathbb E_{\L_i}[h]
=\int\Phi\,d\L(\beta_1^{(n)},\dots,\beta_i^{(n)}),
\qquad
\mathbb E_{\L_{i-1}}[h]
=\int\Phi\,d\left[\L(\beta_1^{(n)},\dots,\beta_{i-1}^{(n)})
\otimes\L(\beta_i^{(n)})\right].
$$
The two laws differ precisely by the correlation between the $\sigma$-algebras
$\mathcal A_1^{i-1}:=\sigma(\beta_1^{(n)},\dots,\beta_{i-1}^{(n)})$
and $\mathcal A_i^i:=\sigma(\beta_i^{(n)})$. These are generated by cylinder sets separated by $w_n$ symbols, so the $\psi$-mixing condition \hyperref[cond: A1+]{(A1+)} gives
$$
d_{\mathrm{\mathrm{TV}}}\left(\L(\beta_1^{(n)},\dots,\beta_i^{(n)}),
\L(\beta_1^{(n)},\dots,\beta_{i-1}^{(n)})\otimes\L(\beta_i^{(n)})\right)
\le\psi(w_n)\lesssim e^{-\eta' (\log n)^2}.
$$
Since $\|\Phi\|_\infty\le 1$,
$$
\left|\mathbb E_{\L_i}[h]-\mathbb E_{\L_{i-1}}[h]\right|\le\psi(w_n),
$$
and taking the supremum over $h$ yields
$$
d_{\mathrm{\mathrm{TV}}}(\L_i,\L_{i-1})\le\psi(w_n)\lesssim n^{-3}.
$$
Summing over $i$ and applying the triangle inequality,
\begin{equation}\label{eqn: mixing bound}
d_{\mathrm{\mathrm{TV}}}\left(
\L(\beta_1^{(n)},\dots,\beta_{k_n}^{(n)}),
\bigotimes_{i=1}^{k_n}\L(\beta_i^{(n)})
\right)
\le\sum_{i=1}^{k_n}d_{\mathrm{\mathrm{TV}}}(\L_i,\L_{i-1})
\le k_n\psi(w_n)\lesssim n^{-2}.
\end{equation}
This estimate is uniform in $\zeta$.

\medskip
\noindent\textbf{Step 3(c): Transferring back to the original block variables.}

By the triangle inequality,
$$
d_{\mathrm{\mathrm{TV}}}\left(
\L(\beta_1,\ldots,\beta_{k_n}),
\bigotimes_{i=1}^{k_n}\L(\beta_i)
\right)
\le T_1+T_2+T_3,
$$
where
\begin{align*}
T_1&:=d_{\mathrm{\mathrm{TV}}}\left(
\L(\beta_1,\ldots,\beta_{k_n}),
\L(\beta_1^{(n)},\ldots,\beta_{k_n}^{(n)})
\right),\\
T_2&:=d_{\mathrm{\mathrm{TV}}}\left(
\L(\beta_1^{(n)},\ldots,\beta_{k_n}^{(n)}),
\bigotimes_{i=1}^{k_n}\L(\beta_i^{(n)})
\right),\\
T_3&:=d_{\mathrm{\mathrm{TV}}}\left(
\bigotimes_{i=1}^{k_n}\L(\beta_i^{(n)}),
\bigotimes_{i=1}^{k_n}\L(\beta_i)
\right).
\end{align*}

\emph{Bound on $T_1$.} The two vectors differ only if, for some $j\le n$ and some $l\le L$, the orbit hits the symmetric difference $A_l(\zeta)\triangle A_{l,n}(\zeta)$. Hence, by \eqref{eqn: cylinder approximation},
$$
T_1
\le\mu\left(\exists\,j\le n,\ l\le L:\,
\I_{A_l(\zeta)}(T^j x)
\ne\I_{A_{l,n}(\zeta)}(T^j x)\right)
\lesssim(\log n)n^{-2}.
$$

\emph{Bound on $T_2$.} This is the mixing term, bounded in \eqref{eqn: mixing bound}: $T_2\lesssim n^{-2}$.

\emph{Bound on $T_3$.} For products of marginals, the total-variation distance is bounded by the sum of the marginal distances:
$$
T_3
\le\sum_{i=1}^{k_n}d_{\mathrm{\mathrm{TV}}}\left(
\L(\beta_i^{(n)}),\L(\beta_i)
\right)
\le\sum_{i=1}^{k_n}\mu(\beta_i^{(n)}\ne\beta_i).
$$
Each $\beta_i$ depends on the orbit over the block $B_i$ of length $b_n$. The two differ only if the orbit hits a symmetric difference at some time in $B_i$. Hence
$$
\mu(\beta_i^{(n)}\ne\beta_i)
\le\sum_{j\in B_i}\sum_{l=1}^L
\mu\left(A_l(\zeta)\triangle A_{l,n}(\zeta)\right)
\le b_n L C\gamma_n.
$$
Summing over $i=1,\ldots,k_n$, and as $k_nb_n\asymp n$ and $\gamma_n=n^{-3}$, as we chose $L=\log n$,
$$
T_3
\le k_n\cdot b_n\cdot L\cdot C\gamma_n\lesssim(\log n)n^{-2}.
$$

Combining the three bounds, for any $\zeta\in \X$,
$$
d_{\mathrm{\mathrm{TV}}}\left(
\L(\beta_1,\ldots,\beta_{k_n}),
\bigotimes_{i=1}^{k_n}\L(\beta_i)
\right)
\le T_1+T_2+T_3
\lesssim(\log n)n^{-2}.
$$
Averaging over the partition $\U_m$,
\begin{equation}\label{eqn: block-decorrelation-averaged}
\sum_{\Delta\in\U_m}\mu(\Delta)\,
d_{\mathrm{\mathrm{TV}}}\left(
\L(\beta_1^{\Delta},\ldots,\beta_{k_n}^{\Delta}),
\bigotimes_{i=1}^{k_n}\L(\beta_i^{\Delta})
\right)
\lesssim(\log n)n^{-2}.
\end{equation}

\noindent\textbf{Step 3(d): Passage to the sums.}
Lastly, let $\widetilde\beta_1,\ldots,\widetilde\beta_{k_n}$ be independent random vectors with $\widetilde\beta_i\overset{d}{=}\beta_i$. As the map $(x_1,\ldots,x_{k_n})\mapsto\sum_{i=1}^{k_n}x_i$ is measurable and  contracting with respect to the metric $d_{\mathrm{\mathrm{TV}}}$, by \eqref{eqn: block-decorrelation-averaged}
$$
\sum_{\Delta}\mu(\Delta)d_{\mathrm{\mathrm{TV}}}\left(
\L\left(\sum_{i=1}^{k_n}\beta^\Delta_i\right),
\L\left(\sum_{i=1}^{k_n}\widetilde\beta^\Delta_i\right)
\right)
\lesssim(\log n)n^{-2}.
$$

\textbf{Step 4: Multivariate Poisson approximation.}
Now, suppress dependence on $\Delta\in\U_m$ again, define $\beta_i$ for an arbitrary $\zeta\in\X$ analogously as in step 3.
For each block $i$, set $p_i^{(k,l)}=\P(\beta_i=e_{k,l})$ and $p_i=\sum_{k=1}^K\sum_{l=1}^L p_i^{(k,l)}$. Since a block can contain a relevant point only when $T^jx\in B(\zeta,r_n(z_0))$ for some $j\in B_i$, so again by Markov's inequality, for all $i$, up to a uniform constant that depends only on $\rho$ and $z_0$,
$$
p_i\le \P\left[\sum_{j\in B_i}\I_{B(\zeta,r_n(z_0))}\ge1\right]
\le\sum_{j\in B_i}\mu(B(\zeta,r_n(z_0)))
\lesssim\frac{b_n}{n}.
$$
Therefore
$$
\sum_{i=1}^{k_n}p_i^2\le
k_n\max_{i}p_i^2\lesssim k_n\frac{b_n^2}{n^2}
\lesssim
\frac{(\log n)^3}{n}.
$$
By the multivariate \hyperref[thm: Le Cam]{Le Cam theorem},
$$
d_{\mathrm{TV}}
\left(
\L\left(\sum_i\beta_i\right),
\L((Z^{(n)}_{k,l})_{k,l})
\right)
\lesssim
\sum_i^{k_n} p_i^2
\lesssim
\frac{(\log n)^3}{n},
$$
where the $Z_{k,l}$ are independent Poisson random variables with
$$
\E Z_{k,l}=\lambda^{(n)}_{k,l}=\sum_{i=1}^{k_n}p_i^{(k,l)}.
$$
We now calculate these means. First, notice that
$$p_{i}^{(k,l)}=\P(S_{i,k,l}=1,\, N_i\le 1)=\P(S_{i,k,l}=1)-\P(S_{i,k,l}=1,\, N_i \ge2),$$
and by the formula $\E[X]=\sum_{q=0}^\infty \P(X\geq q)=\sum_{q=1}^\infty q\P(X=q)$ for all integer-valued non-negative random variables,
$$\P(S_{i,k,l}=1)=\E_\mu[S_{i,k,l}]- \E_\mu\left[S_{i,k,l}\I_{\{S_{i,k,l}\ge2\}}\right],$$
we have
$$\left|p_i^{(k,l)}-\E_\mu[S_{i,k,l}]\right|\le \E_\mu\left[S_{i,k,l}\I_{\{S_{i,k,l}\ge2\}}\right]+\P(S_{i,k,l}=1,\, N_i \ge2).$$
The latter term is trivially bounded by $\P(N_i\ge2)$ and for the first term, notice that
$$
S_{i,k,l}\I_{\{S_{i,k,l}\ge2\}}\le S_{i,k,l}(S_{i,k,l}-1)=\sum_{j\ne j'\in B_i\cap I_k}\I_{A_{l}(\zeta)}\circ T^j\cdot \I_{A_l(\zeta)}\circ T^{j'},
$$
hence we can bound the expectation of such product by
$$\E_\mu[S_{i,k,l}(S_{i,k,l}-1)]\le \sum_{j\neq j' \in B_i}\mu\left(T^{-j}B(\zeta,r_n(z_0))\cap T^{-j'}B(\zeta,r_n(z_0))\right)$$
which was calculated in Step 2 already. Therefore, summing over all $k_n$ blocks we get
$$
\left|\lambda_{k,l}^{(n)}-\sum_{i=1}^{k_n}\E_\mu[S_{i,k,l}]\right|\lesssim \frac{(\log n)^3}{n}.
$$
Obviously, for all $k$ and all $l$, $\sum_{i=1}^{k_n}S_{i,k,l}$ and $M_{n,k,l}^\zeta$ differ in at most $\min\{\#I_k,k_n\ell_n\}$ terms, so
$$\left|\E_\mu\left[\sum_{i=1}^{k_n}S_{i,k,l} \right]-\E_\mu[M_{n,k,l}]\right|\lesssim \frac{\#I_k}n.$$

It remains to compare $n\mu(A_l(\zeta))$ with $2\rho(\zeta)\int_{z_{l-1}}^{z_l}e^{-z}\,dz$. Define
$$
F_n(z,\zeta):=n\mu\left(B\left(\zeta,\frac{e^{-z}}n\right)\right),\qquad z\in[a,b]
$$
where $[a,b]$ is contained in the support of the $g$ in its $z$-coordinate.
Then $n\mu(A_l(\zeta))=F_n(z_{l-1},\zeta)-F_n(z_l,\zeta)$. By the logarithmic continuity assumption \hyperref[cond: LC]{(LC)} on $\rho$, uniformly for $z\in[a,b]$ and for $\mu$-a.e.\ $\zeta$,
$$F_n(z,\zeta)=2\rho(\zeta)e^{-z}+\mathcal E_n(z,\zeta),$$
where $|\mathcal E_n(z,\zeta)|\lesssim 1/\log n$. Moreover, $\mathcal E_n(\cdot,\zeta)$ has bounded variation:
$$
\mathrm{Var}_{[a,b]}\mathcal E_n(\cdot,\zeta)\lesssim\frac1{\log n}.
$$
Indeed, by \hyperref[cond: LC]{(LC)}
$$
\left|\frac d{dz}\mathcal E_n(z,\zeta)\right|=\left|-e^{-z}\left[\rho\left(\zeta+\frac{e^{-z}}n\right)+\rho\left(\zeta-\frac{e^{-z}}n\right)-2\rho(\zeta)\right]\right|\lesssim \frac1{\log n},
$$
Hence
$$
\sum_{l=1}^L\left|n\mu(A_l(\zeta))-2\rho(\zeta)\int_{z_{l-1}}^{z_l}e^{-z}\,dz\right|
=\sum_{l=1}^L|\mathcal E_n(z_{l-1},\zeta)-\mathcal E_n(z_l,\zeta)|\lesssim\frac1{\log n}.
$$
Multiplying by $t_k-t_{k-1}$ and summing over $k$ gives
$$
\sum_{k=1}^K\sum_{l=1}^L\left|\lambda^{(n)}_{k,l}-2\rho(\zeta)(t_k-t_{k-1})\int_{z_{l-1}}^{z_l}e^{-z}\,dz\right|\lesssim\frac1{\log n}.
$$
Thus, if $Z_{k,l}$ are independent Poisson random variables with means $\lambda_{k,l}=2\rho(\zeta)(t_k-t_{k-1})\int_{z_{l-1}}^{z_l}e^{-z}\,dz$,
then
$$
d_{\mathrm{TV}}\left(
\L\left((Z^{(n)}_{k,l})_{k,l}\right),
\L\left((Z_{k,l})_{k,l}\right)
\right)\le\sum_{k,l}|\lambda^{(n)}_{k,l}-\lambda_{k,l}|
\lesssim\frac1{\log n}.
$$
This bound is uniform in $\zeta$.

Combining the four steps for fixed $\zeta$, and averaging over $\Delta\in\U_m$, for $Z^\Delta_{k,l}$ each with mean $\lambda^\Delta_{k,l}$
\begin{equation*} 
\sum_{\Delta\in\U_m}\mu(\Delta)\,
d_{\mathrm{TV}}\left(
\L\left((M_{n,k,l}^{\Delta})_{k,l}\right),
\L\left((Z^\Delta_{k,l})_{k,l}\right)
\right)
\lesssim
\frac1{\log n}+\frac{(\log n)^3}{n}+\frac1{\log n}\lesssim\frac1{\log n}.\qedhere
\end{equation*}
\end{proof}

\begin{proof}[Proof of Proposition~\ref{prop: uniform convergence for localised point processes}]
Recall the notation convention: a quantity carrying a $\Delta$ superscript is evaluated at $\zeta=\zeta_\Delta$ which is the centre of a partition set $\Delta\in\U_m$. One without $\Delta$ is evaluated at an arbitrary $\zeta\in\X$.

For each $\Delta\in\U_m$, let $Z^\Delta_{k,\ell}$ be the independent Poisson random that appeared in Lemma~\ref{lemma: TV approximation} with expectation
$$
\lambda_{k,\ell}^{(n),\Delta}
=\sum_{i=1}^{k_n}p_i^{(k,\ell)}(\zeta_\Delta)=\P\left(S^\Delta_{i,k,l}=1\right)-\P\left(S^\Delta_{i,k,l}=1,\, N^\Delta_i \ge2\right).
$$

For any two point processes $\xi_1,\xi_2$ and any measurable function $h$ on the space of point measures,
\begin{equation}\label{eq: TV-to-expectation}
\left|\E[h(\xi_1)]-\E[h(\xi_2)]\right|
\le \|h\|_\infty\cdot d_{\mathrm{TV}}\left(\L(X),\L(Y)\right).
\end{equation}
Take $\xi_1=N^\Delta_n$ and $\xi_2$ a point process whose law coincides with $\L((Z^\Delta_{k,\ell})_{k,\ell})$, and set
$$
F(\xi):=\exp\left(-\int g\,d\xi\right),
$$
for some non-negative $C^1$ function with compact support (hence $\|F\|_\infty\le 1$).
By Lemma~\ref{lemma: TV approximation}, as $N_n^\Delta$ is a combination of $M_{n,k,l}^\Delta$,
\begin{equation}\label{eq: averaged-TV-consequence}
\sum_{\Delta\in\U_m}\mu(\Delta)\,
\left|\E_\mu\left[e^{-\int g\,dN^\Delta_n}\right]-\E\left[e^{-\sum_{k,\ell} g(t_{k-1},z_{l-1})Z^\Delta_{k,\ell}}\right]\right|
\lesssim\frac1{\log n}.
\end{equation}

For independent Poisson random variables $Z^\Delta_{k,\ell}$ with means $\lambda_{k,\ell}^{(n),\Delta}$, a standard computation (for example, see \cite{Res87}) gives
$$
\E\left[\exp\left(-\sum_{k,\ell} g(t_{k-1},z_{l-1})\,Z^\Delta_{k,\ell}\right)\right]
=\prod_{k,\ell}\exp\!\left(-\lambda_{k,\ell}^{(n),\Delta}\left(1-e^{-g(t_{k-1},z_{l-1})}\right)\right).
$$
Taking logarithms,
$$
\log\E\left[\exp\left(-\sum_{k,\ell} g(t_{k-1},z_{l-1})\,Z^\Delta_{k,\ell}\right)\right]
=-\sum_{k,\ell}\lambda_{k,\ell}^{(n),\Delta}\left(1-e^{-g(t_{k-1},z_{l-1})}\right).
$$

By the proof of Lemma~\ref{lemma: TV approximation}, uniformly in $\zeta\in\X$ up to a measure zero set,
$$
\lambda_{k,\ell}^{(n)}
=2\rho(\zeta)(t_k-t_{k-1})\int_{z_{\ell-1}}^{z_\ell}e^{-z}\,dz
+\mathcal O\!\left(\frac1{\log n}\right).
$$
Applying this at $\zeta=\zeta_\Delta$, substituting into the logarithm of the expectation above, as $n\to\infty$, up to an error term $\lesssim 1/\log n$ that is uniform in $\Delta$,
\begin{equation*}
-\sum_{k,\ell}\lambda_{k,\ell}^{(n),\Delta}\left(1-e^{-g(t_{k-1},z_{l-1})}\right)\\
\to
-2\rho(\zeta_\Delta)\sum_{k,\ell}(t_k-t_{k-1})\int_{z_{\ell-1}}^{z_\ell}
\left(1-e^{-g(t_{k-1},z_{l-1})}\right)e^{-z}\,dz.
\end{equation*}
The double sum is a Riemann sum for the integral
$$
-2\rho(\zeta_\Delta)\int_0^1\int_{-\infty}^{\infty}
\left(1-e^{-\widetilde g(t,z)}\right)e^{-z}\,dz\,dt.
$$
Hence, uniformly in $\Delta$,
\begin{equation}\label{eq: poisson-to-psi}
\left|\E\left[e^{-\sum_{k,\ell} g(t_{k-1},z_{l-1})Z^\Delta_{k,\ell}}\right]-\Psi_\infty^g(\zeta_\Delta)\right|
\lesssim\frac1{\log n}.
\end{equation}

For the final step, by the triangle inequality,
\begin{align*}
&\sum_{\Delta\in\U_m}\mu(\Delta)\,
\left|\E_\mu\left[e^{-\int g\,dN^\Delta_n}\right]-\Psi_\infty^g(\zeta_\Delta)\right|\\
&\le
\sum_{\Delta\in\U_m}\mu(\Delta)\,
\left|\E_\mu\left[e^{-\int g\,dN^\Delta_n}\right]-\E\left[e^{-\sum_{k,\ell} g(t_{k-1},z_{l-1})Z^\Delta_{k,\ell}}\right]\right|\\
&\quad+
\sum_{\Delta\in\U_m}\mu(\Delta)\,
\left|\E\left[e^{-\sum_{k,\ell} g(t_{k-1},z_{l-1})Z^\Delta_{k,\ell}}\right]-\Psi_\infty^g(\zeta_\Delta)\right|.
\end{align*}
By \eqref{eq: averaged-TV-consequence}, the first sum is $\lesssim1/\log n$. By \eqref{eq: poisson-to-psi}, the second term $\lesssim\frac1{\log n}\sum_{\Delta\in\U_m}\mu(\Delta)
=\frac1{\log n}.$
Therefore
\begin{equation*}
\sum_{\Delta\in\U_m}\mu(\Delta)\,
\left|\E_\mu\left[e^{-\int g\,dN^\Delta_n}\right]-\Psi_\infty^g(\zeta_\Delta)\right|
\lesssim\frac1{\log n}
\qedhere
\end{equation*}
\end{proof}

Finally, by Proposition~\ref{prop: uniform convergence for localised point processes}, the formula $\E_\mu[e^{-\int gd\widetilde N_n}]=\sum_{\Delta\in\U_m}\mu(\Delta)\E_\mu[e^{-\int g dN_n^\Delta}]$, and Proposition~\ref{prop: convergence of point process integral}, we may conclude with the following.
\begin{proposition}
For all $C^1$ non-negative test functions with compact support  $g=g(t,z)$, for $\mu$-almost every $\zeta$ that occurs as the centre of $\Delta\in\U_m$, we have 
    $$\lim_{n\to\infty}\E_\mu\left[e^{-\int g dN_n}\right]=\int\exp\left(-\iint \left(1-e^{-g(t,z)}\right)2\rho(x)e^{-z}dzdt\right)\rho(x)\,dx.$$
\end{proposition}
This validates Theorem~\ref{thm: convergence to Cox}.

\section{Proof of Theorem~\ref{main thm: convergence of PP}}\label{sec: proof of record point process convergence}

We will prove Theorem~\ref{main thm: convergence of PP} from
Theorem~\ref{thm: convergence to Cox} via applying the Continuous Mapping Theorem \cite{Bil68}, following the machinery comprehensively established in \cite[\S4]{Res87} and \cite[\S5]{HolTod19}. We give an extremely brief introduction here. 
Let $(M,d)$ and $(M',d')$ be two metric spaces, and a mapping $h:M\to M'$ that is almost surely continuous, which we define below. We call an element $\xi\in M$ a discontinuity with respect to $h$ if there exists a sequence $\{\xi_n\}_n$ such that  $d(\xi_n,\xi)\to0$ as $n\to\infty$ while $d'(h(\xi_n),h(\xi))\not\to 0$. The map is a.s.\ continuous if the set of discontinuities has zero $\mu$ mass. 
For completeness, we can take $d,d'=d_v$ which is compatible with the vague topology and can be constructed via the following: take a countable dense base set $\{f_n\}_n\subset C_c^1\left((0,1]\times\R\right)$, and define the metric by
$$d_v(\mu_1,\mu_2)=\sum_{j=1}^\infty\frac1{2^j}\min\left\{1,\left|\mu_1(f_j)-\mu_2(f_j)\right|\right\}. 
$$ 
It is known that $\mu_n\to\mu$ vaguely if and only if $d_v(\mu_n,\mu)\to0$ as $n\to\infty$. Applying the CMT will then imply that $h(\xi_n) \xrightarrow[]{d'} h(\xi)$ for all continuity points. Let us first recall the precise statement of CMT.

\begin{theorem}[Continuous Mapping Theorem, \cite{Bil68}] Let $(M_1,d_1)$ and $(M_2,d_2)$ be metric spaces and $\xi_n,\xi$ are random elements on $(M_1,\mathcal B(M_1))$. Suppose that $\xi_n\to\xi$ weakly. If a function $h:M_1\to M_2$ satisfies
$$\P\left(\xi\in\left\{x_0\in M_1: x_0\text{ is a discontinuity point for }h\right\}\right)=0
$$
then $h(\xi_n)\to h(\xi)$ weakly as well.
\end{theorem}

In the proceeding proof of Theorem~\ref{main thm: convergence of PP}, we will take $M_{1}=\mathcal M_p\left((0,1]\times \R\right)$ and $M_2=\M_p(\R)$, where $\M_p(X)$ denotes the set of locally finite measures on a space $X$ and is equipped with the vague topology. 
A notable difference between the results presented in this paper and those of \cite{HolTod19} is that we are not seeking to establish a convergence of the partial maxima process $M_n(t)$ to a continuous classical extremal process $Y_G(t)$ generated by one of the three canonical extremal distribution $G$ introduced in \S\ref{ssec: iid evl}. In fact, it should be clear from \cite{Col01} that the distribution of partial maximum of the return process $Y_j(x):=-\log d(T^jx,x)$ converges in limit to a averaged-Gumbel law with density $g_x(z)=\rho(x)e^{-2\rho(x)z}$.

The following lemma verifies the conditions for applying the CMT. Define the  functions which retrieve the behaviour of records: let $E=(0,1]\times\R$, and $\xi=\sum_{i}\delta_{(t_i,z_i)}$, let
$$H_1:\M_p(E)\to \mathbb D(\R),\qquad H_1(\xi)(t):=\sup\{z_i:\, t_i\le t\};$$
$$H_2:\M_p(E)\to \mathbbm D(\R),\qquad H_2(\xi)(z):=\inf\{t_i:\, z_i> z\};$$
$$H_3:\mathbbm D(\R)\to \M_p(\R),\qquad H_3(y):=\sum_{s:y(s^-)\ne y(s^+)}\delta_s.$$
Both $h_1:=H_3\circ H_1$ and $h_2:=H_3\circ H_2$ are functions from $\M_p(E)$ to $\M_p(\R)$. Our first step is to verify that $h_1, h_2$ are almost surely continuous with respect to $\mathcal N$.

\begin{lemma}\label{lemma: continuity record maps}
Let $E=(0,1]\times\mathbb R$ and let $\mathcal N$ the limiting Cox point process as in Theorem~\ref{thm: convergence to Cox}. Let $D_R$, $D_V$ be the set of discontinuities for $h_1$ and $h_2$ respectively. Then
$$
\mathbb P\left(
\mathcal N\in D_R
\right)=0
\quad\text{ and }\quad
\mathbb P\left(
\mathcal N\in D_V
\right)=0.
$$

\end{lemma}

\begin{proof}
For convenience, we recall the Laplace functional of $\mathcal N$
$$
\mathbb E[e^{-\int gd\mathcal N}]=\int_X\exp\left(-\int_E(1-e^{-g(t,z)})2\rho(x)e^{-z}\,dt\,dz\right)\rho(x)\,dx ,
$$
For Borel sets $A,B\subset E$, the
\emph{second factorial moment measure} of $\mathcal N $ is
\begin{equation}\label{eq:second-factorial}
M^{(2)}(A\times B)
=\E\left[\sum_{i\ne j}\mathbf 1_{\{X_i\in A\}}\mathbf 1_{\{X_j\in B\}}\right]=\int_{\X}\Lambda_x(A)\Lambda_x(B)\,\rho(x) dx.
\end{equation}
where the sum in the middle is over ordered pairs of distinct points $\delta_{X_i}=\delta_{(t_i,z_i)}\in \mathcal N$. 
The second identity follows from the Laplace functional with $g=\lambda \I_A+\lambda'\I_B$, then differentiating     against $\lambda,\lambda'$ at $\lambda,\lambda'=0$.
In our case, as $\Lambda_x(d t\,d z)=2\rho(x)e^{-z}d t\,dz$, the product measure $\Lambda_x\otimes\Lambda_x$ has density
$4\rho(x)^2 e^{-z-z'}dtdz\,dt'dz$ which is absolutely continuous with respect to the Lebesgue measure on $\R^4$.

Define the set of pairs with the same time coordinate:
$$
D_{\mathrm{time}}=\{(t,z,t',z')\in E\times E: t=t'\}\subset \R^4.
$$
This set has Lebesgue measure zero in \(\R^4\), because it is contained
in a three dimensional hyperplane of points satisfying $t=t'$. By absolute continuity of 
$\Lambda_x\otimes\Lambda_x$  with respect to the Lebesgue measure, we conclude that
$$
(\Lambda_x\otimes\Lambda_x)(D_{\mathrm{time}})=0
\qquad\text{for every }x.
$$
Substituting into \eqref{eq:second-factorial},
$$
M^{(2)}(D_{\mathrm{time}})
=\int_{\X}(\Lambda_x\otimes\Lambda_x)(D_{\mathrm{time}})\,\rho(x) dx
=0.
$$
Define random variable
$$
S_t:=\sum_{i\ne j}\mathbf 1_{\{t_i=t_j\}}
$$
that counts the number of ordered pairs of distinct points of $\mathcal N$
with the same $t$-coordinate. Its expectation is
$$
\E[S_t]=M^{(2)}(D_{\mathrm{time}})=0.
$$
Since $S_t\ge 0$ is integer-valued, $\E[S_t]=0$ implies
$S_t=0$ almost surely by Markov's inequality. Therefore, almost surely, no two distinct atoms
of $\mathcal N$ share the same $t$-coordinate. The proof for almost surely no two atoms have the same $z$-coordinate follows verbatim.

Next, we show that $\mathcal N$ is locally finite. Let $K\subset E$ be a compact set, then by definition, as $\|\rho\|_\infty<\infty$
$$
\E[\mathcal N(K)]=\int\Lambda_x(K)\rho(x)dx<\infty.
$$ 
As $\mathcal N(K)$ is non-negative, this entails $\P(\mathcal N(K)<\infty)=1$. To prove local finiteness, take a countable family of compact sets $\{ K_n\}_n$ such that their union covers $E$ and
$$\P\left(\left(\bigcap_{n}\left\{\mathcal N(K_n)<\infty)\right\}\right)^c\right)\le \sum_{n}\P\left(\mathcal N(K_n)=\infty\right)=0.
$$
This is sufficient to let us conclude with $$\P(\mathcal N(K)<\infty \text{ for all compact }K\subset E)=1.$$

In particular, for every fixed $t^*\in(0,1]$ and $z^*\in \R$, we can show that $\P(\mathcal N(\{t^*\}\times \R)\ge 1)=0$ and $\P(\mathcal N(0,1]\times\{z^*\})\ge1)=0$ by direct computation of the expectation then apply Markov's inequality. This means there is no accumulation of atoms in either the $t$ or $z$ coordinate, \emph{i.e.}\ almost surely 
$$\mathcal N(\{t^*\}\times \R)=0\quad\text{ and }\quad \mathcal N((0,1]\times\{z^*\})=0.$$

By \cite[\S4]{Res87} and \cite[\S5]{HolTod19} what we have proved so far is enough to conclude that $\P(\mathcal N\in D_R)=0$. The proof for $\P(\mathcal N\in D_V)=0$ follows verbatim.
\end{proof}

By Lemma~\ref{lemma: continuity record maps} we are allowed to apply CMT to conclude that $R_n\to h_1(\mathcal N)$ and $V_n\to h_2(\mathcal N)$ both in distribution. It only remains to identify the limiting distribution, \emph{i.e.}\ the distribution of $h_1(\mathcal N)$ and $h_2(\mathcal N)$.

\begin{proof}{Proof of Theorem~\ref{main thm: convergence of PP}}
By definition of $\mathcal N$, the law of $\mathcal N$ has the form $\P(\mathcal N(\cdot)\le k)=\int \P(\Pi_x(\cdot)\le k)\rho(x)dx$ for all $k\ge0$, where $\Pi_x$ is a Poisson random measure on $E$ with intensity $2\rho(x)e^{-z}dtdz$ (in other words, its intensity measure is $\mathrm{Leb}\times\lambda_{G_x}$ where $G_x(z)=2\rho(x)e^{-z}$). By measurablility of the map $h_1$, for all Borel subset $A\subset (0,1]$ (without loss of generality take $A=(a,b]$)
\begin{equation}\label{eqn: record time Cox limit}\P\left((h_1(\mathcal N))(A)\le k\right)=\int_X \P((h_1(\Pi_x))(A)\le k)\rho(x) dx
\end{equation}
where by classical results on record statistics via Poisson point processes \cite{ResRub73}\footnote{The original statement in \cite{ResRub73} was stated for homogeneous Poisson point process, where $\Pi_x$ is inhomogeneous. However, this can be resolved without much difficulty by introducing a change of variable.}, the number of points in $h_1(\Pi_x)$ on $(a,b]$ is Poisson distributed with mean $\log (b/a)$. As this is independent of $x$, the right-hand side of \eqref{eqn: record time Cox limit} after integrating over $x$ equals $\P(\mathcal R(A)\le k)$, where $\mathcal R$ is a Poisson point process with intensity $1/tdt$. By Lemma~\ref{lemma: continuity record maps} and the CMT, $h_1(N_n)$ converges in distribution to $h_1(\mathcal N)$, hence we conclude that the limiting distribution of $\sum_{j}\delta_{j/n}\I_{\{Y_j>M_{j-1}\}}$ converges to that of a Poisson process with intensity $1/t dt$.

Similarly, using an analogous argument, for any Borel set $B\subset \R$,
$$
\P(h_2(\mathcal N)(B)\le k)=\int \P(h_2(\Pi_x)(B)\le k)\rho(x)dx.
$$ 
By \cite[\S5]{HolTod19}, the distribution of $h_2(\Pi_x)$ has intensity measure
\begin{align*}
\lambda_{V}[a,b]&=-\log (-\log G_x(b))+\log (-\log G_x(a))=-\log \left(-\log (e^{-2\rho(x)e^{-b}})\right)+\log \left(-\log (e^{-2\rho(x)e^{-a}})\right)\\
&=b-\log (2\rho(x))-a+\log (2\rho(x))=b-a
\end{align*}
which is again independent of $x$. Then, the limiting distribution of $h_2(N_n)$ equals that of a Poisson point process $\mathcal V$ on $\R$ with intensity $\lambda[a,b]=b-a$.
\end{proof}

\section{Remarks on different observables}
The main results presented in \S\ref{sec: main results} deal with the observable defined for the value of return distances that results in the averaged-Gumbel type for the limiting distribution of the maximum process. The analogous statement for Frechet-type holds. Take $Y_j(x)=d(x,T^jx)^{-\alpha}$ and define the associated point process $N_n=\sum_{j=1}^n \delta_{\left(j/n,\, n^{-\alpha }Y_j\right)}$. Then 
$$
N_n\Rightarrow \mathcal N^F,\text{  $\mathcal N^F$ a Cox process with random intensity $\Lambda_x(dtdz)=2\rho(x){\alpha}z^{-(1+\alpha)}dtdz$}.
$$
The record-value process $V_n$ converges in distribution to a PRM with intensity $\lambda_F$ on $(0,\infty)$ given by $\lambda_F[a,b]=\alpha \log (b/a)$.
For the Weibull-type, we would define $Y_j(x)=C-d(x,T^jx)^{\alpha}$  and $N_n=\sum_{j=1}^n\delta_{\left(j/n,\, n^{\alpha}(Y_j-C)\right)}$. Then 
$$N_n\Rightarrow \mathcal N^W, \text{$\mathcal N^W$ a Cox process with random intensity $\Lambda_x(dtdz)={2\rho(x)}{\alpha}(-z)^{\alpha-1}dtdz$},
$$ and $V_n$ converges in distribution to a PRM with intensity $\Lambda_W$ on satisfying $\lambda_W[a,b]=\alpha\log(b/a)$ as well.
The limiting distribution of the record-time process $R_n$ is universal for all three cases, \emph{i.e.}\ PRM with intensity measure $1/tdt$.

The proofs of these Fr\`echet and Weibull observable cases follow almost verbatim from those of the Gumbel case without needing to change any assumptions or the proof structure. For the Fr\`echet case: the proof of Proposition~\ref{prop: convergence of point process integral} relies on using compactness of the support of the test function $g$ in its $z$ coordinate, to argue that the difference of 
$$g(j/n,Y_j/n^\alpha)-g(j/n,Y^\Delta_j/n^\alpha)$$ is not automatically zero only when $x$ is in some $E_j(r_n)$ set corresponding to return distance $r_n$. Note that $Y_j(x)/n^\alpha\in[a,b]$ implies $d(T^jx,x)\in[1/nb^{1/\alpha},1/na^{1/\alpha}]$, and $Y_j^\Delta(x)/n^\alpha\in[a,b]$ implies $d(T^jx,x)\in[1/nb^{1/\alpha}-\delta,1/na^{1/\alpha}+\delta]$ where $2\delta$ is the diameter of the partition sets $\Delta\in \U_m$. Similarly for the Weibull case, $n^{\alpha}(Y_j(x)-C)\in[a,b]\subset (-\infty,0]$ implies $d(T^x,x)\in[(-b)^{1/\alpha}/n,(-a)^{1/\alpha}/n]$, and similarly for $n^\alpha(Y_j^\Delta-C)$. Consequently, 
$$\|g(j/n,d(\zeta,T^j\cdot)^{-\alpha})\|_{L^1(\mu)}\lesssim n^{-1}\text{ and }
\|g(j/n,\,n^\alpha(d(\zeta,T^j\cdot)^{\alpha}-C))\|_{L^1(\mu)}\lesssim n^{-1}
$$ remain unchanged.
The proofs for convergence of the Laplace functionals for the auxiliary process are analogous to that of Proposition~\ref{prop: uniform convergence for localised point processes},
requiring only changing the parameters of the Poisson variables and the diameters of the annuli sets.

\section*{Disclaimer of AI usage}
This paper used generative AI tools marginally. In particular, no paid plans features were ever employed. The main usage of a large language model (Deepseek) was for
\begin{enumerate}
\item facilitating literature research (including the literature on Cox process, the precise reference of multivariate Le Cam theorem and the checking the precise requirements to apply the Continuous Mapping Theorem), 
\item spotting typographical errors, and 
\item keeping consistency of my notations throughout (especially for the proof of Prop~\ref{prop: convergence of point process integral} and Prop~\ref{prop: uniform convergence for localised point processes})
\end{enumerate}
There were 2 places where LLM did have an impactful contribution to my proofs:
\begin{enumerate}
    \item  during Step 2 of Lemma~\ref{lemma: TV approximation}, where AI pointed out the error terms of the Riemann sum can be dealt with via Lipschitz continuity;
    \item helping me finish the mixing argument in Step 3(b) of the proof of Lemma~\ref{lemma: TV approximation}.
\end{enumerate}
% \section*{Acknowledgement}
% The author would like to thank Prof. Mike Todd for his comments on \S~\ref{sec: proof of record point process convergence}.


\begin{thebibliography}{XXX}
\bibitem[BJK]{BarJurKlo23} B. B\'ar\'any, N. Jurga, and I. Kolossv\'ary, \emph{On the convergence rate of the chaos game}, Int. Math. Res. Not. IMRN \textbf{(5)} (2023), 4456–4500.

\bibitem[Bar]{Bar88} A. Barbour, \emph{Stein’s method and Poisson process convergence}, J. Appl. Probab.\textbf{ 25A}  (1988) 175–184.

\bibitem[BS]{BarSou01}  L. \ Barreira and B. \ Saussol, \emph{Hausdorff dimension of measures via Poincar\'e recurrence}, Comm. Math. Phys., \textbf{219} (2001) 443–463.

\bibitem[Bil]{Bil68} P. Billingsley, \emph{Convergence of Probability Measures}, Wiley, New York (1968).

\bibitem[Bow]{Bow75} R. Bowen, \emph{Equilibrium States and the Ergodic Theory of Anosov Diffeomorphisms}, Springer, Berlin Heidelberg (1975).

\bibitem[CC]{ChaCol13} J. Chazottes and P. Collet, \emph{Poisson approximation for the number of visits to balls in non-uniformly hyperbolic dynamical systems}, Ergodic Theory Dynam.
Systems \textbf{33} (2013) 49–80.

\bibitem[Col]{Col01} P. \ Collet, \emph{Statistics of closest return for some non-uniformly hyperbolic systems} Ergodic Theory and Dynamical Systems, \textbf{21}(2) (2001) 401-420.

\bibitem[Cox]{Cox55} D. Cox, \emph{Some statistical methods connected with series of events}, J. R. Stat. Soc. \textbf{17}(2) (1955) 129–164. 

\bibitem[FFV]{FreFreVai17} A.C.M. Freits, J. M. Freitas and S. Vaienti, \emph{Extreme Value Laws for non stationary processes generated
by sequential and random dynamical systems}, Annales de l’Institut Henri Poincaré - Probabilités et Statistiques \textbf{53}(3) (2017) 1341-1370.

\bibitem[GHN]{GupHolNic11} C. Gupta, M. Holland and M. Nicol, \emph{Extreme value theory and return time statistics for dispersing billiard maps and flows, Lozi maps and Lorenz-like maps}, Ergodic Theory and Dynam. Systems. \textbf{31}(5) (2011) 1363-1390.

%\bibitem[HY]{HayYan19} N. Haydn and F. Yang:, \emph{Local escape rates for $\phi$-mixing dynamical systems}, Ergod. Th. \& Dynam. Sys. \textbf{40}(10) (2020) 2854–2880.

\bibitem[HNT]{HolNicTor12} M. Holland, M. Nicol and A. T\"or\"ok, \emph{Extreme value distributions for non-uniformly expanding dynamical systems}, Trans. Am. Math. Soc. \textbf{364} (2012) 661-688.

\bibitem[HKKP]{HKKP} M. \ Holland, M. \ Kirsebom, P. \  Kunde,  and T. Persson, \emph{Dichotomy results for eventually always hitting time statistics and almost sure growth of extremes}, Trans. Amer. Math. Soc. 377 (2024) 3927-3982.

\bibitem[HT1]{HolTod19} M. \ Holland and M. \ Todd, \emph{Weak convergence to extremal processes and record events for non-uniformly hyperbolic dynamical systems}, Ergodic Theory Dynam. Systems \textbf{39} (2019) 980-1001

\bibitem[HT2]{HolTod25} M. \ Holland and M. \ Todd, \emph{On distributional limit laws for recurrence}, Nonlinearity \textbf{38}, (2025) no.  075028.

\bibitem[HS]{HolStr69} P. Holmes and W. Strawderman, \emph{A note on the waiting times between record observations}, J. Appl. Probab. \textbf{6}(3) (1969) 711-714. 


\bibitem[Kal]{Kal83} O. Kallenberg. Random Measures.  Akademie-Verlag, Berlin, 3. 1983.

\bibitem[Kar]{Kar66} S. Karlin, \emph{A first course in stochastic processes},  Academic Press, New York, 1966.

\bibitem[KKP]{KirKunPer23} P. \ Kirsebom, M. \ Kunde and T. \ Persson, \emph{On Shrinking Targets and Self-Returning Points}, Ann. Sc. Norm. Super. Pisa Cl. Sci., \textbf{24}, (2023) 1499-1535.


\bibitem[LP]{LasPen17} G. Last and M. Penrose, \emph{Lectures on the Poisson Process}. Cambridge University Press (2017). 

\bibitem[LLR]{LLR80} M. Leadbetter, G. Lindgren and H. Rootz\'en, \emph{Extremes and Related Properties of Random Sequences
and Processes} (Springer Series in Statistics). Springer, New York (1980).

\bibitem[LCam]{LCam60} L. Le Cam, \emph{An approximation theorem for the Poisson binomial distribution}, Pacific J. Math.\textbf{10} (1960) 1181-1197.



\bibitem[R1]{Res73} S. Resnick, \emph{Limit laws for record values}, J. Stoch. Process and Their Appl. \textbf{1} (1973) 67-82.

\bibitem[R2]{Res87} S.\ Resnick, \emph{Extreme Values, Regular Variation and Point Processes}. Springer, New York \textbf{1} (1987). 

\bibitem[RR]{ResRub73} S.\ Resnick and M. \ Rubinovitch,\emph{The structure of extremal processes}, Advances Appl. Probab. \textbf{5} (1973) 287-307.


\bibitem[Sou1]{Sau04} B. \ Saussol, \emph{Recurrence rate in rapidly mixing dynamical systems}, Discrete and Continuous Dynam. Systems, \textbf{15} (2004) 259-267. 

\bibitem[Sou2]{Sau06} B. Saussol, \emph{Recurrence rate in rapidly mixing dynamical systems}, Discrete and Continuous Dynamical Systems, \textbf{15}(1) (2006) 259-267.

\end{thebibliography}
\end{document}